\documentclass[pdflatex,sn-mathphys-num]{sn-jnl}

\usepackage{graphicx}%
\usepackage{multirow}%
\usepackage{amsmath,amssymb,amsfonts}%
\usepackage{amsthm}%
\usepackage{mathrsfs}%
\usepackage[title]{appendix}%
\usepackage{xcolor}%
\usepackage{textcomp}%
\usepackage{manyfoot}%
\usepackage{booktabs}%
\usepackage{algorithm}%
\usepackage{algorithmicx}%
\usepackage{algpseudocode}%
\usepackage{listings}%

\usepackage{standalone}
\usepackage{tikz}
\usetikzlibrary{calc, shapes.misc,positioning, decorations.pathmorphing, arrows.meta}
\usepackage{subcaption}
\usepackage[compatibility=false]{caption}
\usepackage{bm}
\usepackage{bbm}

\theoremstyle{thmstyleone}%
\newtheorem{theorem}{Theorem}
\newtheorem{proposition}[theorem]{Proposition}%
\theoremstyle{thmstyletwo}%
\newtheorem{remark}{Remark}%
\theoremstyle{thmstylethree}%

\def\pathfig{.}
\DeclareMathAlphabet{\itbf}{OML}{cmm}{b}{it}
\DeclareMathAlphabet\mathbfcal{OMS}{cmsy}{b}{n}

\renewcommand{\hat}{\widehat}
\renewcommand{\tilde}{\widetilde}
\def\RR{\mathbb{R}}
\def\bx{{{\itbf x}}}
\def\bxi{\boldsymbol{\xi}}

\def\be{{\itbf e}}

\def\bQ{{\itbf Q}}

\def\bmu {{\boldsymbol{\mu}}}

\def\bphi{\boldsymbol{\phi}}

\def\be#1\ee{%
    \begin{equation}%
    #1%
    \end{equation}%
}

\def\bphi{{\boldsymbol{\varphi}}}

\def\12{{\frac{1}{2}}}

\def\gam{{\boldsymbol \gamma}}
\def\Beta{{\boldsymbol \beta}}

\newcommand{\CC}{{\mathbb{C}}}

\newcommand\hgam{\boldsymbol{\hat\gamma}}
\renewcommand\gam{\boldsymbol{\gamma}}
\renewcommand\bphi{\boldsymbol{\phi}}
\renewcommand{\bxi}{\boldsymbol \xi}

\newcommand{\Alpha}{\boldsymbol \alpha}
\renewcommand{\Beta}{\boldsymbol \beta}

\newcommand{\hKappa}{\hat{\boldsymbol \kappa}}

\renewcommand{\Xi}{\boldsymbol \xi}
\newcommand\blF{ {\cal F}}

\newcommand\blTF{ {\tilde {\cal F}}}
\newcommand\blC{{\cal C}}

\newcommand\blU{{\cal U}}

\newcommand{\blUi}[2]{\blU_{#1}^{#2}}

\DeclareFontFamily{U}{mathx}{\hyphenchar\font45}
\DeclareFontShape{U}{mathx}{m}{n}{
 <5> <6> <7> <8> <9> <10>
 <10.95> <12> <14.4> <17.28> <20.74> <24.88>
 mathx10
 }{}
\DeclareSymbolFont{mathx}{U}{mathx}{m}{n}
\DeclareFontSubstitution{U}{mathx}{m}{n}
\DeclareMathAccent{\widecheck}{0}{mathx}{"71}
\DeclareMathAccent{\wideparen}{0}{mathx}{"75}

\def\clr{\color{black}}

\begin{document}

\title[Lanczos for continuous spectra]{Adapting the Lanczos algorithm to matrices with almost continuous spectra}


\author*[1]{\fnm{J\"orn} \sur{Zimmerling}}\email{jorn.zimmerling@it.uu.se}

\author[2,3]{\fnm{Vladimir} \sur{Druskin}}\email{vdruskin@gmail.com}
\equalcont{These authors contributed equally to this work.}

\affil[1]{\orgdiv{Institutionen f\"or informationsteknologi}, \orgname{Uppsala Universitet}, \orgaddress{\street{Regementsv\"agen 10}, \city{Uppsala}, \postcode{751 20}, \state{Uppsala}, \country{Sweden}}}

\affil[2]{\orgdiv{Department of Mathematical Sciences}, \orgname{Worcester Polytechnic Institute}, \orgaddress{\street{100 Institute Road}, \city{Worcester}, \postcode{01609}, \state{MA}, \country{USA}}}

\affil[3]{\orgdiv{Department of Mathematics}, \orgname{Southern Methodist University}, \orgaddress{\street{3100 Dyer St, Clements Hall Room 208}, \city{Dallas}, \postcode{75205}, \state{TX}, \country{USA}}}


\abstract{\unboldmath
We consider the approximation of $B^T (A+sI)^{-1} B$, where
$A\in\mathbb{R}^{n\times n}$ is large and symmetric positive definite and
$B\in\mathbb{R}^{n\times p}$ with $p\ll n$. {\clr We are interested in the case
where $A$ has an almost continuous (dense) spectrum, i.e. its eigenvalues fill one or
more intervals so densely that a Krylov method cannot resolve them
individually.} Our target application is the computation of multiple-input multiple-output transfer functions arising from large-scale discretizations of problems with continuous spectral measures, such as linear time-invariant PDEs on unbounded domains.  Traditional Krylov methods, such as Lanczos or conjugate gradients, focus on resolving individual eigenvalues of a dense discretization, while ignoring the underlying continuous spectral measure that these points approximate. We argue that it is more efficient to model the inherent branch cut of the original operator than to exhaustively resolve the artificial point spectrum induced by discretization. {\clr For that purpose we adapt the framework} of Kreĭn--Nudelman semi-infinite
strings to the block Lanczos algorithm, with parameters chosen adaptively by maximizing the energy absorbed at the string termination relative to the energy stored in the string.  This approach results in a low-rank modification to the (block) Lanczos matrix, dependent on $\sqrt{s}$, with an additional $O(n)$ cost.   We demonstrate significant error reductions for large-scale self-adjoint PDE discretizations in unbounded domains, including two- and three-dimensional Maxwell's equations in diffusive regimes. The method proves particularly advantageous in computing state-space solutions for wave propagation, specifically for 2D wave and 3D Maxwell operators.  Implicitly replacing the conventional Lanczos spectral decomposition with a representation in terms of the continuous Kreĭn--Nudelman spectrum, we obtain a qualitative improvement in finite-difference approximations, effectively transforming standing-wave artifacts into outgoing propagating waves.
}

\keywords{Block Lanczos, acceleration, quadrature, Kreĭn--Nudelman, square-root terminator, unbounded domains, absorbing boundary conditions}

\pacs[MSC Classification]{65F10, 65N22, 65F50, 41A20, 30B70, 47B36, 78M20}

\maketitle

\section{Introduction}

\subsection{Problem statement}

Let $A=A^T\in\RR^{n\times n}$ be a symmetric positive definite (s.p.d.) matrix and let $B$ be a tall matrix with $B\in\RR^{n\times p}$, with $p\ll n$. We refer to $p$ as the size of the block. We target the approximation of  the square multi-input multi-output (MIMO) transfer function
\be \label{eq:prob1}
\blF(s) =B^T(A+sI)^{-1}B,
\ee
for $s\in\CC\setminus (-\infty, 0)$ outside the spectrum of $A$.  Generally, $\blF(s)\in\CC^{p\times p}$ is a complex symmetric matrix-valued function and it becomes s.p.d. for real positive $s$.
We compute the matrix  form~\eqref{eq:prob1}
via the block Krylov subspace 
\[
\mathcal{K}_m(A,B):= \operatorname{blkspan}\{ B,\, AB,\, A^2B,\, \dots,\, A^{m-1}B \},
\]
where $ \rm blkspan$ means that the whole space range $([\, B, AB, A^2 B , \dots, A^{m-1} B\,])$ is generated.
To simplify notation, without loss of generality, we assume that the matrix $B$ has orthonormal columns.
Conventionally, the (block) Krylov subspace approximations $\blF_m$ of matrix  forms~\eqref{eq:prob1} are computed via (block) Gau{\ss}ian quadratures, e.g., \cite{GM10,RRT16,lot2013}. The efficiency of such an approach follows from the shifted energy norm optimality of the Lanczos approximations for $s\in\RR_+$ and  near-optimality for  $f(A)B$ for general $f$ \cite{amsel2023nearoptimal}  yielding adaptation to coarse nonuniform spectral distributions; however, this property weakens on intervals of dense uniform spectra. An error improvement of about one order of magnitude for a given $m$ was obtained in \cite{zimmerling2025monotonicity} by averaging block Gau{\ss} and block Gau{\ss}--Radau quadratures. Combinations of such quadratures were used for tight two-sided bounds \cite{lot2008,Meurant2023,zimmerling2025monotonicity,Lun18}, such that averaging naturally leads to a more accurate estimate.

{\clr
Throughout we work in the three-stage setting
\[
\text{operator } A_\infty \ \longrightarrow\ \text{discretization } A
\ \longrightarrow\ \text{Lanczos approximation},
\]
illustrated in Figure~\ref{fig:three_boxes}. Here $A_\infty$ is an operator with
continuous spectrum, either a PDE operator on the unbounded domain, or its
finite-difference analogue on an infinite grid, in which case $A_\infty$ is
bounded. Passing from $A_\infty$ to the matrix $A$ does two distinct things.
Fixing the step size restricts the spectral interval, but leaves the
spectrum continuous. Replacing the infinite exterior part of the grid by an
exponentially convergent finite grid, known as an optimal grid or
finite-difference Gau{\ss}ian quadrature \cite{idkGrids,Druskin2016}, then makes
the spectrum discrete. Only the second step destroys spectral continuity.

The operators we have in mind are self-adjoint, with a spectrum containing one or more intervals of continuous spectrum, possibly together with finitely many
isolated eigenvalues arising from inhomogeneities of the coefficients. We do not
consider compact operators, whose eigenvalues accumulate only at the origin,
nor operators with singular continuous spectrum.

Domain truncation replaces the continuous part of the spectrum of $A_\infty$ by dense eigenvalues of $A$. We call such a spectrum \emph{almost continuous},
or \emph{dense} for short, meaning that on the spectral intervals of interest
the eigenvalue spacing lies far below the resolution attainable in $m\ll n$
Lanczos steps. The Lanczos approximation then cannot separate individual
eigenvalues there and effectively sees only the underlying continuous spectral
measure.

Both discretization steps contribute to the error in approximating
$\blF_\infty$, the transfer function of $A_\infty$, and they behave differently
near the continuous spectrum. The step size fixes an interval of validity in
$s$, whereas the domain truncation controls the error at a given $s$ and, with
optimal grids, decays exponentially in the number of exterior points, with sharp
available bounds \cite{idkGrids,Druskin2016}. Since the Krylov approximation $\blF_m$ approximates $\blF$ rather than $\blF_\infty$, no Krylov
method can reduce the error below the level set by these two steps. We therefore
assume throughout that the grid has been chosen so that $\blF$ approximates
$\blF_\infty$ to the desired tolerance at the values of $s$ of interest. This floor is reached soonest for $s$ close to the continuous spectrum, where both the discretization error and the Krylov approximation error are largest, and it limits the acceleration.

\begin{figure}[ht]
\centering
\resizebox{\textwidth}{!}{%
\begin{tikzpicture}[
    box/.style={
        draw,
        rectangle,
        text width=0.27\textwidth,
        minimum height=1.5cm,
        align=center,
        inner sep=6pt
    },
    boxD/.style={
        draw,
        rectangle,
        text width=0.38\textwidth,
        minimum height=1.5cm,
        align=center,
        inner sep=6pt
    },
    arrow/.style={
        -{Stealth[scale=1.3]},
        very thick
    },
    wavyarrow/.style={
        {Stealth[scale=1.3]}-{Stealth[scale=1.3]},
        thick,
        decorate,
        decoration={snake, amplitude=1.5pt, segment length=6pt, pre length=4pt, post length=4pt}
    }
]
\node[box] (box1) {Operator with continuous\\ spectrum $A_\infty$};
\node[box, right=0.02\textwidth of box1, yshift=-1.5cm] (box2) {Discretization:\\ Matrix $A \approx A_\infty$ with ``almost\\ continuous spectrum''};
\node[boxD, right=0.02\textwidth of box2, yshift=1.5cm] (box3) {Kre\u{\i}n--Nudelman damped Lanczos approximation of $A$ with continuous spectrum};
\draw[arrow] (box1.south) -- (box2.west);
\draw[arrow] (box2.east) -- (box3.south);
\draw[wavyarrow] (box3.west) -- (box1.east);

\end{tikzpicture}
}
\caption{\clr  Schematic of the three-stage workflow for computing PDE transfer functions.
Here we focus on the third stage as a stand-alone procedure. It exploits the spectral continuity of $A_\infty$ to construct reduced-order models (ROMs) for $A$ by collectively approximating clusters in its dense spectrum via continuous spectral distributions, thereby accelerating Lanczos approximations of the transfer functions of $A$. As a separate application to wave problems in unbounded domains, the same approach replaces standing waves---artifacts of nonabsorbing domain truncation---with outgoing waves, yielding a qualitatively more faithful approximation to the solution associated with
$A_\infty$. The full realization of the unified three-stage pipeline, culminating in computations straight for $A_\infty$, remains an active direction of ongoing and future work.}
\label{fig:three_boxes}
\end{figure}

Gau{\ss} and Gau{\ss}--Radau quadratures associated with $\mathcal{K}_m(A,B)$ produce symmetric nonnegative-definite Krylov projection matrices (positive definite in the Gau{\ss} case) matching, respectively, $2m$ and $2m-1$ Stieltjes matrix moments of the matrix-valued spectral measure of $A$ associated with $B$\footnote{\clr Gau{\ss}--Radau quadrature fixes one prescribed node, at either end of the spectrum, at the cost of one matched moment. For discretizations of pure diffusion operators the lowest eigenvalue is set by the domain size and approaches $0$ as the domain grows; for operators with an exact null space, like the Maxwell operator, $0$ is already an eigenvalue of $A$. Either way, we place the prescribed Gau{\ss}--Radau node there.}.
{\clr Here the matrix-valued spectral measure $\boldsymbol\mu$ associated with the pair
$(A,B)$ is supported on the eigenvalues $\lambda_j$ of $A$ with matrix weights
$(B^Tv_j)(B^Tv_j)^T$, where the $v_j$ are the corresponding orthonormal
eigenvectors, so that
\be\label{eq:measmom}
\blF(s)=\int\frac{d {\boldsymbol\mu}(\lambda)}{\lambda+s},
\qquad
{\boldsymbol \xi}_k:=\int\lambda^k\,d{\boldsymbol\mu}(\lambda)=B^TA^kB,\quad k=0,1,2,\dots \ .
\ee
The ${\boldsymbol \xi}_k\in\RR^{p\times p}$ are the Stieltjes matrix moments and matching $2m$ of
them is exactly what $m$ block Lanczos steps achieve, since ${\boldsymbol \xi}_k$ is determined by $\mathcal{K}_m(A,B)$ for $k\le 2m-1$. The nodes of $\boldsymbol\mu$ are fixed by $A$ alone, but its weights depend on how $B$ is distributed over the eigenvectors, and it is this measure, rather than the spectrum of $A$, that the quadratures approximate.}

{\clr The Ritz values associated with the Gau{\ss} and Gau{\ss}--Radau 
constructions define two matrix-valued quadratures on the real line,
and averaging them improves the approximation by effectively doubling the number
of quadrature nodes~\cite{zimmerling2025monotonicity}. The transfer function they produce is rational (a Pad\'e-type approximant) with all its poles at the negative Ritz values $-\theta_i$, i.e.\ on the spectral interval of $-A$ where the continuous spectrum of $A_\infty$ also lies. Such an approximation can reproduce a continuous measure only by resolving the individual eigenvalues underneath it. This is efficient when the spectrum is coarse and nonuniform, which is the regime in which Lanczos is classically near-optimal, but on an almost continuous spectrum the eigenvalue
spacing is far below what $m\ll n$ steps can resolve, each additional Ritz value
accounts for a vanishing share of the measure, and convergence degrades to the
linear rate observed in section~\ref{sec:NumEx}. The effort is moreover spent on
eigenvalues that are themselves artifacts of the domain truncation rather than
features of $A_\infty$. This suggests instead building spectral continuity
directly into the Lanczos approximation, so that dense clusters are represented
collectively rather than resolved one at a time.}

In this work we develop such approximations using extensions of the Stieltjes
string theory developed by Mark Kreĭn and his school. The derivation of the
block Gau{\ss}--Radau formulas in \cite{zimmerling2025monotonicity} rested on the connection between the block Lanczos algorithm and the block extension of
discrete Stieltjes strings and block Stieltjes continued fractions. The scalar
variant of this connection goes back to Kreĭn in the 1950s, see, e.g.,
\cite{Krein1967,Krein1952,Krein1947}. The Lanczos
recursion gives a string of $m$ point masses, and the Gau{\ss} and
Gau{\ss}--Radau rules correspond to terminating it by a Dirichlet or a Neumann
condition. Both terminations are non-absorbing, i.e. a wave reaching the end of the string is reflected. Kreĭn and Nudelman \cite{KreinNudelman1989} extended
the construction to problems with continuous spectra by attaching a damper at
the end of the string, while still preserving $2m$ (or $2m-1$, depending on the
realization) Stieltjes moments.

{\clr Here we adapt this construction to the block Lanczos case. The damper is a
low-rank, $\sqrt{s}$-dependent modification of the last block of the Lanczos
matrix, so the spectral moments matched by the conventional block Lanczos method
are retained, while the poles of the approximation move off the spectral
interval onto the second Riemann sheet of $\sqrt{s}$. Poorly converged Ritz
values are displaced furthest, so that dense spectral regions are represented by
a smooth measure rather than resolved eigenvalue by eigenvalue. The damper is
chosen by maximizing the energy dissipated at the truncation relative to the
energy stored in the string, which is the quantity governing how far the poles
are displaced and hence the smoothness of the spectral measure. This is
illustrated schematically in Figure~\ref{fig:TikzDrawing}.}

{\it We apply this approach mainly to accelerate the computation of the transfer
function of $A$.} However, we also touch on the underlying problem for
$A_\infty$: for the state representation of wave propagation problems in
unbounded domains, the same modification removes the artifacts of the
non-absorbing finite-difference approximation in the exterior, transforming
standing waves into outgoing ones that propagate along d'Alembert
characteristics. 

\begin{figure}[htbp]
    \centering
    \begin{tikzpicture}[
    pale_pink/.style={fill=red!15},
    pale_blue/.style={fill=blue!15},
    cross/.style={
        cross out,
        draw=black,
        very thick,
        minimum size=10pt,
        inner sep=0pt,
        outer sep=0pt
    },
    blue_cross/.style={
        cross,
        draw=blue
    },
    dashed_arrow/.style={
        -{Stealth[length=3mm]},
        dashed,
        thick
    }
]

\def\xmin{-7}
\def\xmax{7}
\def\ymin{-5}
\def\ymax{5}

\fill[pale_pink] (\xmin, 0) rectangle (0, \ymax);
\fill[pale_pink] (\xmin, \ymin) rectangle (0, 0);
\fill[pale_blue] (0, 0) rectangle (\xmax, \ymax);
\fill[pale_blue] (0, \ymin) rectangle (\xmax, 0);

\draw[-{Stealth[length=3mm]}, thick] (\xmin, 0) -- (\xmax, 0) node[anchor=north east] {\large Re $\sqrt{s}$};
\draw[-{Stealth[length=3mm]}, thick] (0, \ymin) -- (0, \ymax) node[anchor=north west] {\large Im $\sqrt{s}$};



\node[draw, rounded corners, inner sep=5pt, fill=white, align=center] at (-4.5, 4.2) {\large Second Riemann sheet\\ ($\operatorname{Re}( \sqrt{s})<0$)};
\node[draw, rounded corners, inner sep=5pt, fill=white, align=center] at (4.5, 4.2) {\large First Riemann sheet\\ ($\operatorname{Re}(\sqrt{s})>0$)};

\draw[dashed, ultra thick] (0, \ymin) -- (0, \ymax);
\node[rotate=90, anchor=north, inner sep=10pt] at (0, 2) {\large Branch cut};

\def\xpole{-3.5}
\def\xpolemid{-4.0} 
\def\yhi{3.0}
\def\ymid{2.0}      
\def\ylo{1.2}


\node[blue_cross] (b1) at (0, \yhi) {};
\node[cross] (k1) at (\xpole, \yhi) {};
\draw[dashed_arrow] (b1) -- (k1);

\node[blue_cross] (b2) at (0, -\yhi) {};
\node[cross] (k2) at (\xpole, -\yhi) {};
\draw[dashed_arrow] (b2) -- (k2);

\node[blue_cross] (b3) at (0, \ymid) {};
\node[cross] (k3) at (\xpolemid, \ymid) {};
\draw[dashed_arrow] (b3) -- (k3);

\node[blue_cross] (b4) at (0, -\ymid) {};
\node[cross] (k4) at (\xpolemid, -\ymid) {};
\draw[dashed_arrow] (b4) -- (k4);

\node[blue_cross] (b5) at (0, \ylo) {};
\node[cross] (k5) at (\xpole, \ylo) {};
\draw[dashed_arrow] (b5) -- (k5);

\node[blue_cross] (b6) at (0, -\ylo) {};
\node[cross] (k6) at (\xpole, -\ylo) {};
\draw[dashed_arrow] (b6) -- (k6);

\node[inner sep=10pt,align=center] at (4, 0) {\large Gau\ss\ poles on branch cut\\ \phantom{xx}};

\node[text width=6cm, align=center] at (-4, 0) {\large Kreĭn -- Nudelman poles\\ on second Riemann sheet};

\draw[thin, black] (\xmin, \ymin) rectangle (\xmax, \ymax);

\end{tikzpicture}
    \caption{\clr The two Riemann sheets of $z=\sqrt{s}$ are shown. The poles of the Gau{\ss} approximant lie on the imaginary axis, i.e.\ the spectral interval of the
operator $-A$, they are its Ritz values. The Kre\u{\i}n--Nudelman approach moves
these poles to the second Riemann sheet, enlarging the domain of analyticity and
improving convergence.}
    \label{fig:TikzDrawing}
\end{figure}

The remainder of the article is organized as follows: we review the basic properties of the block Gau{\ss} quadrature computed via the block Lanczos algorithm, its connection to Stieltjes strings and representation via truncated Stieltjes-matrix continued fraction in section~\ref{sec:BlLanc}. In section~\ref{sec:KreinNudelman1} we extend block Stieltjes strings/continued fractions to, respectively, Kreĭn--Nudelman strings/continued fractions. We further map back the Kreĭn--Nudelman extension in terms of the Lanczos block tridiagonal matrix. Then we discuss the adaptive choice of the Kre\u{\i}n--Nudelman parameters,
showing that Gau{\ss} and Gau{\ss}--Radau quadratures are two limiting cases
of the Kre\u{\i}n--Nudelman extension, and derive a two-sided error bound for
the latter for $s\in\mathbb{R}^+$. In section~\ref{sec:HermitePade} we develop
an adaptive spectral optimization algorithm to determine the
Kre\u{\i}n--Nudelman parameters based on maximization of the dissipated energy
relative to the stored energy of the Kre\u{\i}n--Nudelman string along the negative real axis. The dissipated energy is related to the density of the
spectral measure of the approximant, so this maximization is equivalently an
optimization of the regularity of that density: the poles of the Gau{\ss}
approximant on $\mathbb{R}^-$ are replaced by a smooth density of the same
total mass.}

\subsection{Notation}\label{sec:note}
Given two square s.p.d. matrices $G_1, G_2$ we use the notation
$G_1<G_2$ to mean that
the matrix $G_2-G_1$ is positive definite (i.e., Löwner ordering). 
A sequence of matrices $\{G_m\}_{m\ge 0}$ is said to be monotonically
increasing (resp. decreasing) if $G_m < G_{m+1}$ (resp. $G_{m+1}<G_m$)
for all $m$.
Definite $p \times p$ matrices are denoted by Greek 
letters $\Alpha,\Beta,\gam$ and matrix-valued functions by calligraphic capital letters such as $\blC(s)$ or $\blF(s)$. 
For $\Alpha, \Beta \in \mathbb{R}^{p\times p}$ and $\Beta$ is nonsingular, we use the notation $\frac{\Alpha}{\Beta}:= \Alpha \Beta^{-1}$ to denote right inversion. We write $\tfrac{1}{A}$ in place of $\tfrac{I_p}{A}$ whenever no ambiguity arises, in order to improve readability, especially in continued fractions.
The matrix $E_k \in \mathbb{R}^{mp\times p}$ has zero elements except for the
$p\times p$ identity matrix in the $k$-th block, $E_k =[0,\ldots, 0,I, \ldots, 0]^T$. Finally, the imaginary unit is denoted as $\imath$ to distinguish it from indices $i,j,k$.

\section{Block Gau{\ss}ian Quadratures, Stieltjes strings and continued fractions}\label{sec:BlLanc}

 {\begin{center}
\begin{minipage}{.65\linewidth}
 \begin{algorithm}[H]
\caption{Block Lanczos iteration}\label{alg:blockLanc}
\begin{algorithmic}
\normalsize
\State Given $m$, $A\in{\mathbb R}^{n\times n}$ s.p.d., 
$B\in{\mathbb R}^{n\times p}$ with orthonormal columns 
\State $Q_1 =B$ \Comment{$\Beta_{1}=I_p$ as $B^TB=I$}
\State $W = AQ_1$
\State $\Alpha_1 = Q_1^T W$
\State $W = W - Q_1 \Alpha_1$
\For{$i= 2,\dots, m$} 
 	\State $Q_i\Beta_{i}=W$ \Comment{QR decomposition of $W$}
	\State $W = AQ_i- Q_{i-1}\Beta_{i}^T$
 	\State $\Alpha_i = Q_i^T W$
 	\State $W = W - Q_i \Alpha_i$
\EndFor 
\end{algorithmic}
 \end{algorithm}
\end{minipage}
\end{center}}
\vspace{0.5cm}
Let us assume that we can perform $m$ steps, with $mp \le n$, of the block Lanczos iteration (Algorithm~\ref{alg:blockLanc}) without breakdowns 
or deflation \cite{OLeary1980,GOLUB1977}. As a result, 
the orthonormal block vectors $Q_i\in \mathbb{R}^{n \times p}$ form 
the matrix ${\boldsymbol Q}_m=[Q_1, \ldots, Q_m]\in \mathbb{R}^{n\times mp}$, whose columns
contain an orthonormal basis for the block Krylov subspace
\[\mathcal{K}_m(A,B):= \operatorname{blkspan}\{ B,\, AB,\, A^2B,\, \dots,\, A^{m-1}B \}.\]
{In addition, let us assume that after $m$ steps one can compute the QR decomposition  $Q_{m+1}\Beta_{m+1}=W$ with full rank, where $W$ is computed on the $m$-th step. 
The Lanczos recursion (with $m$ multiplications by $A$)} can then be compactly written as
\be\label{eq:LancRel}
A{\boldsymbol Q}_m={\boldsymbol Q}_m T_m + Q_{m+1} \Beta_{m+1}E_m^T, 
\ee
where $T_m$ is the s.p.d. block tridiagonal matrix 
\be\label{eq:T}
T_m=
\begin{pmatrix}
\Alpha_1 	& \Beta_2^T 	& {}		&{}			& {}& {}& {}\\
\Beta_2 	& \Alpha_2	& \Beta_3^T	&{}			& {}& {}& {}\\
{}			& \ddots 	& \ddots 	& \ddots 	& {}& {}& {}\\
{}			& {}		& \Beta_{i}& \Alpha_i & \Beta_{i+1}^T & {}& {}\\
{}			& {}		&{}			& \ddots 	& \ddots 	& \ddots& {}\\
{}			& {}		&{}			& {}	& \Beta_{m-1}& \Alpha_{m-1}& \Beta_m^T\\
{}			& {}		&{}			& {} 	& {} 	& \Beta_m& \Alpha_m\\
\end{pmatrix} ,
\ee
and $\Alpha_i, \Beta_i \in\mathbb{R}^{p \times p}$ are the block coefficients
in Algorithm~\ref{alg:blockLanc}. 

Using the Lanczos decomposition, $\blF(s)$ can be approximated as
\be\label{eq:blapprox}
 \blF(s)\approx \blF_m(s)= E_1^T(T_m+sI)^{-1}E_1.
\ee
This approximation is known as a block Gau{\ss} quadrature \cite{GM10}. {\clr Such a quadrature interpretation is closely related to the classical Stieltjes moment problem and to continued-fraction representations of
Stieltjes functions.  In the scalar case, the $m$-step Lanczos recursion
generates the Jacobi matrix $T_m$, whose eigenvalues are the nodes of the
$m$-point Gaussian quadrature rule, while the squared first components of
the corresponding eigenvectors give the quadrature weights.  The resulting
quadrature is exact for polynomials of degree up to $2m-1$, or,
equivalently, it matches the first $2m$ moments of the underlying spectral
measure.  For the Stieltjes transform, this same construction yields a $[m-1,m]$ rational approximation, which is a Pad\'e approximation at infinity and can be represented by a finite Stieltjes continued fraction. The coefficients of this continued fraction are precisely the parameters of the Stieltjes--Kreĭn string that we describe below and are a key part of the proposed approach.  This connection between Gau{\ss}ian quadrature, moment matching, Jacobi matrices, Stieltjes--Kreĭn string and Stieltjes continued fractions is classical; see, in particular, Supplement~II in Gantmacher and Kreĭn~\cite{GantmacherKrein2002} and
\cite{LiesenStrakos2012}.  The block construction above is the
matrix-valued analogue: $T_m$ is block Jacobi, the scalar spectral measure
is replaced by the matrix-valued measure associated with $B$, and
\eqref{eq:blapprox} matches the $2m$ matrix moments
$\{B^T A^i B\}_{i=0}^{2m-1}$  (see e.g. \cite{zimmerling2025monotonicity} for definitions).

It is well known that the simple Lanczos algorithm, which relies solely on three-term recurrences without reorthogonalization, is numerically unstable due to floating-point round-off errors. This instability manifests as a loss of orthogonality among the Lanczos vectors and the emergence of spurious copies of the Lanczos eigenvalues. However, this instability merely slows the convergence of the transfer function to some degree; it does not compromise the accuracy of the final converged result \cite{Knizhnerman1996TheSL,Greenbaum1989}. Equivalent theoretical guarantees for the block Lanczos algorithm, however, are currently unknown. Furthermore, for very large-scale problems ($N > 10^7$), we observed unexplained deviations from scalar Paige's theory, which serves as the foundation for the aforementioned convergence results. In these extreme cases, the stability of the final result is lost, necessitating the use of local reorthogonalization.}

To interpret the system as a block extension of Stieltjes-Kreĭn  strings, we apply a block $LDL^T$ decomposition to $T_m$ (as defined in equation~\eqref{eq:T}), following the approach established in \cite{zimmerling2025monotonicity}. This results in
\begin{equation}\label{eq:defT}
T_{m} :=(\widehat{K}_{m}^{-1})^T {J}_{m} \Gamma^{-1}_m {J}_{m}^T \widehat{K}_{m}^{-1}
\end{equation} 
where $\hKappa_{i}\in\RR^{p\times p}$, $\hKappa_1=I_p$ and $\gam_{i}\in\RR^{p\times p}$ are all full rank, 
\[
{J}_{m}^T = 
\begin{bmatrix}
I_p & -I_p & ~ & ~ \\
~ & \ddots& \ddots & ~ \\
 ~& ~& \ddots & -I_p \\
 ~ & ~&~& I_p 
\end{bmatrix}\in\RR^{pm\times pm},
	\begin{array}{ll}
\widehat{K}_{m}&={\rm blkdiag}(\hKappa_{1},\dots,\hKappa_{m})\\
{\Gamma}_m&={\rm blkdiag}(\gam_{1},\dots,\gam_{m}) ,
\end{array}
\]
and
$\Alpha_1=(\hKappa_{1}^{-1})^T\gam_1^{-1}\hKappa_{1}^{-1}=\gam_1^{-1}$, 
$\Alpha_i=(\hKappa_{i}^{-1})^T(\gam_{i-1}^{-1}+\gam_{i}^{-1})\hKappa_{i}^{-1}$ and 
$\Beta_i= { -(\hKappa_{i}^{-1})^T \gam_{i-1}^{-1} \hKappa_{i-1}^{-1}}$ for $i=2,\ldots,m$. 

The matrices $\gam_i$ and $\hKappa_{i}$ can be computed directly during the block Lanczos recursion using the coefficients $\Alpha_i$'s and $\Beta_i>0$'s using Algorithm~\ref{alg:ExtractGam}, given in Appendix~\ref{ap:A}, and previously derived for a different parametrization in \cite{ZaslavskySfraction}.

We first convert $T_m$ to pencil form using the $LDL^T$ factors introduced in decomposition~\eqref{eq:defT}. To this end, we introduce the matrices $\hgam_j$
\be\label{eq:decompose}
\hgam_j =\hKappa_j^T \hKappa_j, \quad j=1,\ldots,m.
\ee
The matrices $\gam_j$ and $\hgam_j$ are known as the Stieltjes parameters, and they are both s.p.d. if block Lanczos runs without breakdown. While $\hKappa_j$'s depend on the choice of block orthogonalization used in the block Lanczos algorithm, the Stieltjes parameters do not.

Let $ Z_m:= {J}_m {\Gamma}_m^{-1} {J}_{m}^T$ and $\widehat{ 
 \Gamma}_m={\rm blkdiag}(\hgam_{1},\dots,\hgam_{m})$.
Then, due to $B^TB=I_p$, the initial parameter $\hgam_1=I_p$ and the approximation $\blF_m$ of the transfer-function ${\cal F}=B^T (A+sI)^{-1}B$ can be expressed using the pencil $(Z_m,\widehat{\Gamma}_m)$ as
\be\label{eq:Gauss}
\blF_m(s)=E_1^T(T_m+sI)^{-1}E_1= E_1^T( Z_m+s \widehat{ 
 \Gamma}_m)^{-1}E_1.
\ee

Thus 
$\blF_m(s)$ corresponds to the first $p\times p$
block of the solution ${ U}_m\in \mathbb{C}^{mp\times p}$ of the linear system
\be\label{eq:LAform}
( Z_m+s \widehat{\Gamma}_m){ U}_m =E_1,
\ee
with block column vector ${ U}_m = [\blU_1; \ldots; \blU_m]$, with $\blU_i$ a $p\times p$ block.

By augmenting ${U}_m$ with the ``boundary" blocks $\blU_0, \blU_{m+1} \in \mathbb{C}^{p \times p}$ and introducing the fictitious term $\gam_0 = I_p$, we can rewrite \eqref{eq:LAform} as a block (finite-difference) Stieltjes string
\begin{eqnarray}
\frac 1 \gam_0\left(\blU_1 - \blU_0\right) &=&-I_p \label{eqn:line1}\\
\frac 1 {\gam_{i-1}}\left(\blU_{i} -\blU_{i-1}\right) 
 - \frac 1 {\gam_{i}} \left(\blU_{i+1}-\blU_{i}\right)+s\hat\gam_i \blU_i
 &=&\mathbf{0}, \quad
i=1,\ldots, m\label{eq:linei}\\
\blU_{m +1}&=& \mathbf{0}. \label{eqn:linem}
\end{eqnarray}
where the first line can be interpreted as a block Neumann condition and the last as a block Dirichlet condition.

For $p=1$ Kreĭn \cite{Krein1952} interpreted the system (\ref{eqn:line1}-\ref{eqn:linem}) as a so-called Stieltjes string
\be\label{eq:ss} -u_{xx}+s\mathbf{m}_m(x)u=0 \ee on a positive interval $[0,x_{m+1}]$, with boundary conditions $u_x(0)=-1$, $u(x_{m+1})=0$, and discrete mass distribution
\be\label{eq:mass} \mathbf{m}_m(x)=\sum_{i=1}^m\hgam_i\delta(x-x_i),\ee
where $x_1=0$, $x_{i+1}=x_i+\gam_i$, $i=1,\ldots, m$,
with $\blU_i=u(x_i)$ (see \cite{Zimmerling_KreinNudel} for a detailed derivation).
 The  Stieltjes-Kreĭn string interpretation is fundamental for the derivation of our algorithm.

The eigendecomposition of $T_m$ traditionally links $\blF_m(s)$ to block Gau{\ss}ian quadrature \cite{MeurantBook}; however, we interpret it here as the transfer function (Neumann-to-Dirichlet map) of the block Stieltjes string terminated by a block Dirichlet condition. Replacing the termination condition~\ref{eqn:linem} with a block Neumann condition, $\blU_{m+1}-\blU_{m}=\mathbf{0}$, introduces $p$ eigenvalues at $0$ and leads to a Gau{\ss}--Radau quadrature.

\begin{remark}\label{rem:oneLap}[Interpretation of the shift $s$ as Laplace Variable $s$ and $\sqrt{s}$]
The transition between diffusive and wave-type problems is handled through a reinterpretation of the shifted system with shift $s$ as Laplace variable. In the diffusion case, we consider operators of the form $(-\Delta + sI)$, where $s$ is the transform of the time derivative $\partial_t$ and $-\Delta$ the negative Laplacian. For the wave equation, $(-\Delta + \hat{s}^2 I)$, we define $\hat{s} = \sqrt{s}$ to map the second-order temporal derivative $\partial_{tt}$ into the same algebraic framework. This substitution allows us to interpret the same Stieltjes string representation from both perspectives, provided the spectral domain of interest is appropriately mapped.
\end{remark}

\section{Block Kreĭn--Nudelman extension}\label{sec:KreinNudelman1}
Extending this analogy to problems arising from differential equations on a semi-infinite interval $x \in [0,\infty)$, Kreĭn and Nudelman truncated the Stieltjes string using a finite-difference variant of the Sommerfeld radiation condition \cite{KreinNudelman1989}. We write such a condition for $p\ge 1$ in the form
\be\label{eq:sommerfeld}
({\sqrt{s}}{\bphi} ) \blU_{m+1}= - \frac 1 {\gam_{m}} (\blU_{m+1}-\blU_{m}),
\ee
replacing the equation~\eqref{eqn:linem} in the block Stieltjes string.{\renewcommand{\thefootnote}{:}
Here $\bphi\footnotemark[0] \in\RR^{p\times p}$ is some s.p.d. matrix parameter whose choice we discuss later.  

\footnotetext[0]{{\clr We use boldface $\bphi$ to distinguish it from the $p=1$ case.}}
}
Returning to the interpretation of \cite{KreinNudelman1989} for scalar $\phi$ and $p=1$, we extend the problem from the finite interval $[0, x_{m+1}]$ to the semi-infinite domain $[0, \infty)$. To do so, we first replace the mass distribution $\mathbf{m}_m(x)$ from equation~\eqref{eq:mass} with 
\be\label{eq:masssinf} 
\hat{\mathbf{m}}_m(x) = \mathbf{m}_m(x) + \phi^{2}\eta(x - x_{m+1}),
\ee
where $\eta(\cdot)$ denotes the Heaviside step function. This leads to the governing ODE 
\be\label{eq:ssKN} 
-u_{xx} + s\hat{\mathbf{m}}_m(x) u = 0, 
\ee
subject to the boundary condition $u_x(0) = -1$ and the condition at infinity $u(\infty) = 0$ for $s\in\CC\setminus (-\infty, 0)$. 

On the tail interval $[x_{m+1}, \infty)$, this parametrization yields outgoing wave solutions of the form 
$u(x) = \blU_{m+1} e^{-\sqrt{s} x \phi}$, {\clr with $\phi$ a positive scalar}. Consequently, the condition
\be\label{eq:sommerfeldE} 
\phi\sqrt{s} u|_{x=x_{m+1}} = -u_x|_{x=x_{m+1}}
\ee
truncates the semi-infinite domain back to $[0, x_{m+1}]$.  This expression is equivalent to the 1D Sommerfeld radiation condition, while the truncation in equation~\eqref{eq:sommerfeld} serves as its finite-difference approximation using a one-sided derivative.

We denote the resulting leading block of a block Stieltjes {\clr string} corresponding to Equations~(\ref{eqn:line1},\ref{eqn:linem},\ref{eq:sommerfeld})
\[\hat\blF_m^{\bphi}(s)= \blU_1^{\bphi}=u(0)\]
 and call $\hat \blF_m^{\bphi}(s)$ the Kreĭn--Nudelman quadrature.
 
\begin{remark}\label{rem:two}
  It is worth noting, that condition \eqref{eq:sommerfeldE}  is a simplified variant of the two-term  square-root terminator condition \cite{BeerPettifor1984Terminator}, which will be discussed in the conclusion. We chose \eqref{eq:sommerfeldE} for simplicity of incorporating it in the first-order port-Hamiltonian framework used for adaptive choice of $\phi$, see Appendix~\ref{sec:port}. However the two term square-root condition can be incorporated in the Kreĭn--Nudelman framework via discrete external mass distribution  $\hat {\mathbf{m}}_m{\clr (x)}=\mathbf{m}_m{\clr (x)}+\phi^2\sum_{i=0}^\infty h\delta(x_{m+1}+ih-x)$, replacing \eqref{eq:sommerfeldE} with $ \phi\sqrt{s+(0.5hs)^2}u|_{x=x_{m+1}}=-u_x|_{x=x_{m+1}}$, according to formula \eqref{eq:sqrt}.  It can be viewed as a counterpart of \eqref{eq:sommerfeldE} when the exterior domain with mass density $\phi^2$ is discretized using equidistant {\clr finite-difference} grid with step $h$  \cite{druskin2016near}, and   they coincide in the limiting case $h\to 0$.
  This two-term condition may yield a more accurate truncation with proper optimization of the two parameters --- this may be a subject of future research.
\end{remark}

\subsection{Kreĭn--Nudelman extension in terms of the tri-diagonal matrix and connection to Gau{\ss}--Radau quadratures}

In the general case $p\ge 1$ and
using equation~\eqref{eq:sommerfeld} we eliminate $\blU_{m+1}$ from the $m$-th equation of
the block Stieltjes system~\eqref{eq:linei} 
and obtain 
\[\frac{1} {\gam_{m-1}}\left(\blU_{m}^{\bphi} -\blU_{m-1}^{\bphi}\right) 
 - \frac 1 {\gam_{m}} \left(\left[\frac {1} {\gam_{m}}+{\sqrt{s}}{\bphi}\right]^{-1}\frac {1} {\gam_{m}}-I\right)\blU_{m}^{\bphi}+s\hat\gam_m \blU_m=0.
 \]
 Combining this with the remaining equations of the system~\eqref{eq:linei} and transforming the obtained matrix pencil, we obtain an expression for  $\hat\blF_m^{\bphi}(s)$ in terms of a block tridiagonal system
 \be\label{eq:KreinNudelman} 
 \hat\blF_m^{\bphi}(s)=E_1^T(\hat T_m^{\bphi}(s)+sI)^{-1}E_1.
 \ee
 Here $\hat T_m^{\bphi}(s)$ coincides with $T_m$ except for the last block diagonal element
\be\label{eq:hatal}\hat \Alpha_m^{\bphi}(s)=\Alpha_m-\delta \Alpha_m^{\bphi}(s), \quad  \delta \Alpha_m^{\bphi}(s)=(\hKappa_{m})^{-T}\gam_m^{-1}(\gam_m^{-1}+\sqrt{s}\bphi)^{-1}\gam_m^{-1}( \hKappa_{m})^{-1}.\ee
By construction, this yields a symmetric $\hat \Alpha_m$ and  $ \lim_{\bphi\to\infty} \hat\blF^{\bphi}_m(s)=\blF_m(s)$. The other limit corresponds to the Gau{\ss}--Radau quadrature $\blTF_{m}(s)$ given by $ \lim_{{\bphi}\to 0} \hat\blF_m^{{\bphi}}(s)=\blTF_m(s)$ as defined in \cite{zimmerling2025monotonicity}.


\begin{remark}\label{rem:pert}
  Realization of the Kre\u{\i}n--Nudelman extension as a perturbation of the last
  elements of the tridiagonal matrix has a practically important consequence.
  Assume for simplicity the simple Lanczos algorithm with $p=1$, and let the Ritz value
  $\theta_i\in\RR_+$ and vector $w_i\in\RR^m$, $\|w_i\|=1$, satisfy
   \be\label{eq:Ritz} T_mw_i-\theta_i w_i=0.\ee
  Since $\hat T^{\bphi}_m(s)$ depends on $s$, the poles of
  $\hat\blF^{\bphi}_m$ solve a nonlinear eigenvalue problem; to first order,
  however, its argument may be evaluated at the unperturbed pole $s=-\theta_i$.
  Provided $|\delta\Alpha^{\bphi}_m(-\theta_i)|$ is small compared with the gaps
  between the Ritz values, the displacement is then
  \be\label{eq:RitzShift}
  \delta\theta_i \approx -\,\delta\Alpha^{\bphi}_m(-\theta_i)
  \left(e_m^Tw_i\right)^{2},
  \ee
  where $e_m^Tw_i$ is the $m$-th component of $w_i$. Together with the residual
  bound $|\theta_i-\lambda'|\le\beta_{m+1}|e_m^Tw_i|$, in which $\lambda'$ is the
  eigenvalue of $A$ closest to $\theta_i$, this shows that a Ritz value which has
  not converged is displaced, whereas one with $|e_m^Tw_i|$ small is
  both accurate and left essentially in place. The perturbation thus acts
  selectively on the poorly resolved, densely spaced part of the spectrum, and
  since $\delta\Alpha^{\bphi}_m(-\theta_i)$ is complex it carries those Ritz
  values off the real axis onto the second Riemann sheet.
\end{remark}

\noindent In \cite{zimmerling2025monotonicity} it was shown that, as $m$ increases, the
block Gau{\ss} and block Gau{\ss}--Radau approximants increase and decrease,
respectively, in the L\"owner order, and satisfy the two-sided bound, for all
$s\in\RR_+$,
 \be\label{bound:GR} \ldots < \blF_{m-1}(s)<\blF_{m}(s)\le \blF(s)\le\blTF_{m}(s)< \blTF_{m-1}{\clr (s)}<\ldots\,. \ee
The following proposition extends this bound to the Kreĭn--Nudelman quadrature and shows that similar to Gau{\ss} and Gau{\ss}--Radau quadrature, it is a Stieltjes function; however, unlike those, it is not meromorphic. 

\begin{proposition}\label{prop:main}
For $0< {\bphi} <\infty $ 
\begin{enumerate}
\item $\hat \blF^{{\bphi}}_m(s)$ is a Stieltjes function with the branch cut on $\RR_-$;
\item $\forall s\in\RR_+$ we obtain the two sided bound
\be\label{in:KN} \blF_{m}(s) < \hat \blF^{{\bphi}}_m(s)<\blTF_{m}(s).\ee
\end{enumerate}
\end{proposition}
 The proof is given in Appendix~\ref{sec:ProofProp1}. Proposition~\ref{prop:main} states that the exact solution  and the Kreĭn--Nudelman approximation satisfy the same two-sided bounds. 

Finally, in Appendix~\ref{ap:ContFrac} we introduce the block Stieltjes continued-fraction representation \eqref{eq:S-fraction2} of $\hat\blF^{\bphi}_m(s)$ which matches $2m-1$ Stieltjes moments of $\blF (s)$.  The
classical Kreĭn--Nudelman extension \cite{KreinNudelman1989} matches $2m $ moments without additional cost. Our Kreĭn--Nudelman approximation can be adapted accordingly. However, it leads to an additional free matrix valued parameter complicating optimization presented in  section~\ref{sec:HermitePade}.  Its benefits in testing are insignificant, yet, for completeness, we give it in Appendix~\ref{ap:extended}.

\section{Selection of \texorpdfstring{$\bphi$}{bphi} via maximization of energy outflow}\label{sec:HermitePade}
It remains to choose the non-negative matrix parameter ${\bphi}$ that minimizes the ``reflection'' induced by the truncation of the Lanczos recursion, {\clr by shifting the poles of the approximation, $\hat \blF^{{\bphi}}_m(s)$ as far onto the second Riemann sheet as possible}. The main idea is to interpret the system as a Stieltjes string that is truncated by a damper $\bphi$ and to maximize the ratio of dissipated energy to the total stored energy in the system. While this ratio vanishes for both Gau{\ss} and Gau{\ss}--Radau quadratures, an optimal choice of $\bphi$ should yield a non-trivial positive semi-definite value. In Appendix~\ref{sec:port} we link the block Stieltjes string to port-Hamiltonian systems and give a detailed derivation.

First, we use the 1D Kreĭn--Nudelman representation and its energy considerations in choosing ${\phi}$.
Multiplying ODE~\eqref{eq:ssKN} by {\clr the complex conjugate of function $u(x)$ denoted $\overline{u(x)}$} and integrating by parts on the interval $[0,x']$, $x'>0$, we obtain
\be\label{eq:energyxm}
{\clr \hat\blF_m^{\phi}(s)=-u(0)[\overline{u}_x(0)] =\int_0^{x'} [ u_x\overline{u_x}+s\mathbf{m}{\clr (x)}u\overline{u}]dx- \left.(u_x\overline{u})\right|_{x=x'}.}
\ee
For $x'=\infty$  and $s\in\CC\setminus (-\infty, 0)$ the last term vanishes and we  obtain   \[\hat\blF_m^{\phi}(s)= \int_0^{\infty} \left[ |u_x|^2+s\mathbf{m}(x)|u|^2\right]\,dx,\]
which for $p=1$ becomes the energy of the system or a stationary functional.
If $x'=x_{m+1}$ with the help of the radiation condition~\eqref{eq:sommerfeldE}  we obtain
\be\label{eq:energyKN}
\hat\blF_m^{\phi}(s)=\int_0^{x_{m+1}} \left[ |u_x|^2+s\mathbf{m}(x)|u|^2 \right] dx +  \clr {({\sqrt{s}}{\phi})\left.|u|^2\right|_{x=x_{m+1}}}.
\ee
The last term of the RHS is related to the energy outflow due to the boundary truncation at $x_{m+1}$. 
{\clr Maximizing this term relative to the stored energy therefore achieves
maximal damping of the Kreĭn--Nudelman term. The physical intuition is that the
exponential damping rate of each mode is proportional to the negative real part
of its pole. Directly maximizing those real parts, or the area of analyticity of
$\hat\blF_m^{\phi}(s)$, is awkward to formulate as an optimization problem, so
we instead maximize the ratio of energy outflow to stored energy. The energy outflow
equals, up to a positive factor, the density of the spectral measure of
$\hat\blF^{\bphi}_m$ on the branch cut, so the maximization can be seen
an optimization of the regularity of that density.}

{\clr We note that the transfer functions of these second-order problems are related to their first-order counterparts by a factor of $\sqrt{s}$ (See Remark~\ref{rem:oneLap}). For the latter, the energy outflow can be obtained directly from the imaginary part of the transfer function, yet the stored energy cannot. 
A detailed derivation of this relationship and the used objective function using port-Hamiltonian systems \cite{PortHammelVanDer,PortHammilBettie} is provided in Appendix~\ref{sec:port}.}

{\clr Thus in the general $p>1$ case we} consider the following optimization problem 
{\clr
\be\label{eq:Optimize}
\bphi^{\star}
=\arg\max_{{\bphi}\geq0}\ \int_{\cal S}
\frac{\Re\,\mathrm{tr}\!\left(-\dfrac{ \hat\blF^{\bphi}_m(s) }
{\sqrt{s}}\right)}
{\|X^{\bphi}\|^{2}_{F}
+\dfrac{1}{|s|}\,\|X^{\bphi}\|^{2}_{\Re\,\hat T^{\bphi}_m(s)}}\ ds ,
\ee
with the Kre\u{\i}n--Nudelman solution
\[
X^{\bphi}(s) = (T_m - E_m\delta\Alpha_m^{{\bphi}}(s) E_m^T + sI )^{-1} E_1=E_1^T(\hat T_m^{\bphi}(s)+sI)^{-1}E_1,
\]
 and matrix norm $\|X\|^2_{M}=\mathrm{tr}\!\left(X^{H}MX\right) $.
 The numerator of \eqref{eq:Optimize} is the energy outflow,  the denominator is the stored energy with the first and second terms being respectively electric and magnetic energies, if one views the Kreĭn--Nudelman network in terms of RLC network instead of the original mechanical interpretation by Kreĭn and Nudelman \cite{KreinNudelman1973}, as shown in Appendix~\ref{sec:port}, see Figure~\ref{fig:tikzexample}.
 In our implementation, ${\cal S}$ is a contour} enclosing the sub-interval $[-d,\, 0]$ of  $-A$'s spectrum at a distance comparable with the spectral gaps, to avoid numerical problems due to the singularities of  $\hat\blF^{{\bphi}}_m$ at Ritz values. The lower boundary $d$ is chosen so that the spectral measure of $-A$ approximates its exact continuous counterpart sufficiently well on the sub-interval, and  $[-d,\, 0]$  contains a sufficient number ($\gg p^2$) of Ritz values to avoid overfitting. We use  the Nelder–Mead simplex method \cite{LagariasOpt} for the optimization  in our numerical experiments.

The trace in equation~\eqref{eq:Optimize} accounts for the MIMO case, where
$\hat\blF^{{\bphi}}_m$ is matrix-valued. Maximizing the relative energy
dissipation while preserving $2m-1$ of the $2m$ Stieltjes moments matched by the
block Lanczos method shifts the poles of $\blF_m$ (the negative eigenvalues of
$T_m$) to those of $\hat\blF^{{\bphi}}_m$ (the negative Kreĭn--Nudelman
eigenvalues of $\hat T_m^{{\bphi}}(s)$, or of the linear eigenvalue problem
{\clr associated with}~\eqref{eq:Ph_phi}), moving them off the negative real
axis onto the second Riemann sheet ($\Re\sqrt{s}<0$). This effectively increases the domain of analyticity for the approximation, as illustrated in Figure~\ref{fig:lists}. {\clr  The figure shows both Riemann sheets of the Kre\u{\i}n--Nudelman transfer function by representing it as a function of $\sqrt{s}$. Transfer functions associated with PDEs on unbounded domains can be represented in terms of scattering poles located on the second Riemann sheet \cite{Zworski}. This motivates constructing approximations that also have poles on the same sheet. These approximate poles cannot, however, be interpreted as approximations of the true scattering poles. As $m$ increases, they move toward the branch cut, represented here by the imaginary axis. The Kre\u{\i}n--Nudelman approximation is rational in $\sqrt{s}$. The geometric convergence rate of such rational approximations is governed, roughly speaking, by the distance from the approximation region to the nearest pole. It is therefore advantageous to displace the approximate poles as far away as possible, thereby enlarging the domain of analyticity. In contrast, the Gau{\ss}  poles lie on the imaginary axis, so that the analyticity domain of the corresponding approximation is restricted to the first Riemann sheet. }

\begin{figure}[ht!]
  \centering
  \includegraphics[width=\textwidth]{\pathfig/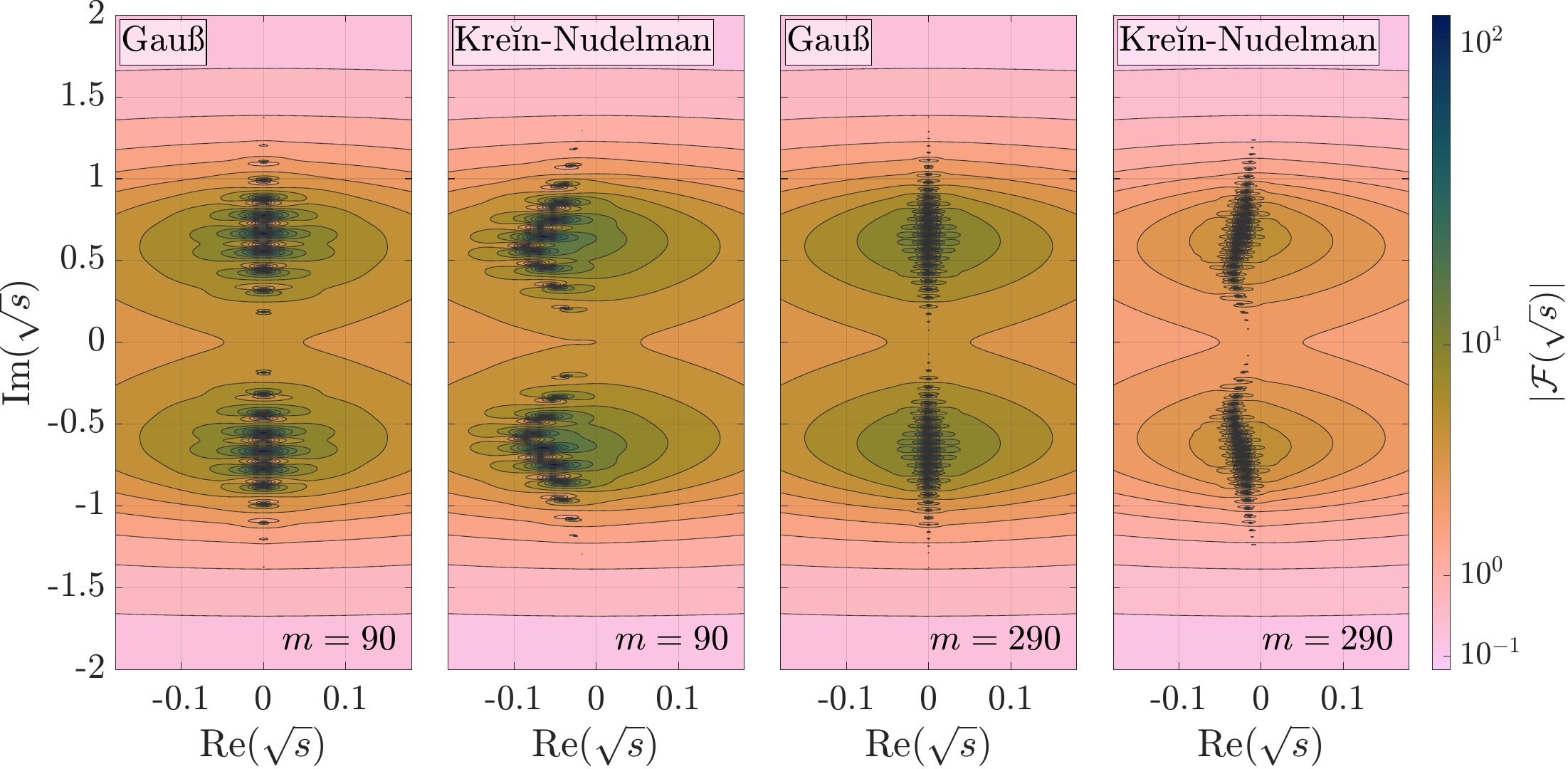}
  \caption{This figure shows a contour plot of respectively $\log |1+\hat \blF^{{\bphi}}_m(s)|$ (Kreĭn--Nudelman)  and $ \log |1+\blF_m(s)|$ (Gau{\ss})  with  $p=1$ (SISO) for $m=90$ (left) and $m=290$ (right)  as a function of $\sqrt{s}$ in the  complex plane for the 3D Maxwell problem from subsection~\ref{sec:EMdiff}. To show both Riemann sheets in a single figure, the contours are shown as a function of the real and imaginary part of $\sqrt{s}$. The  Gau{\ss} quadrature $\blF_m(s)$ is a single-valued function, thus it is even with respect to the imaginary axis, and as a function of $\sqrt{s}$ has poles on the imaginary axis.   Due to the introduction of $\sqrt{s}$ in the definition of the Kreĭn--Nudelman quadrature $\hat \blF^{\bphi}_m(s)$ it becomes a two-valued function (similar to the transfer function for the continuous problem), with the branch cut at the imaginary axis and poles at the second Riemann sheet $\Re(\sqrt {s})<0$. The dissipation optimization moves poles away from the branch cut adaptively; {\clr  i.e., according to Remark~\ref{rem:pert}, (in the perturbative regime) the Ritz values  are shifted  to the second Riemann sheet proportionally to their distance from the true spectrum, thus poorly converged Ritz values in the dense spectral area are moved further than converged well-separated Ritz values.} This distance diminishes for the larger $m$ together with the quadrature approximation error, thus the Kreĭn--Nudelman poles approach the ones of the Gau{\ss}\,  quadrature as $m$ grows. The spectral measure at the branch cut is smooth in the area of poor convergence, collectively approximating the Stieltjes moments at dense intervals of $A$'s spectrum. {\clr These numerical results visually confirm the schematic representation in Figure~\ref{fig:TikzDrawing}. }}
  \label{fig:lists}
\end{figure}

{\clr
Assuming matrix $A$ approximates an operator with a continuous spectrum, $\hat{\blF}^{{\bphi}}_m$ approximates a function characterized by poles on the second Riemann sheet similar to scattering poles \cite{Zworski}.} The optimization improves the convergence until the approximation error becomes comparable to the error introduced by the discretization of the operator in $A$. This balance provides a reasonable stopping criterion. A similar approach proved successful in the optimization of perfectly matched layers \cite{druskin2016near}. While the reasoning presented here is intuitive, a rigorous mathematical justification is the subject of planned future work.

 {\remark During the optimization, we need to compute
\[
X^{\bphi}(s) = (T_m - E_m\delta\Alpha_m^{{\bphi}}(s) E_m^T + sI )^{-1} E_1,
\]
repeatedly for different $\bphi$. Application of the Sherman--Morrison--Woodbury
formula gives
\begin{align*}
X^{\bphi}(s) = &(T_m + sI )^{-1} E_1 \\
&-(T_m + sI )^{-1} E_m
 \left[-\{\delta\Alpha_m^{{\bphi}}(s)\}^{-1}+ E_m^T (T_m + sI )^{-1}E_m\right]^{-1}
 E_m^T(T_m + sI )^{-1}E_1 ,
\end{align*}
and thus we precompute
\begin{align*}
X_m^m(s)&=(T_m + sI )^{-1} E_m, &  X_m^1(s) &=(T_m + sI )^{-1} E_1,\\
F_{m}^{m,m}(s)&=E_m^TX_m^m(s), & F_{m}^{1,m}(s)&=E_m^TX_m^1(s),
\end{align*}
for all quadrature points used to discretize the integral in the optimization
objective~\eqref{eq:Optimize}. Only $p \times p$ systems then have to be solved
repeatedly during the optimization,
\[
X^{\bphi}(s) = X_m^1(s)-X_m^m(s)
\left[-(\delta\Alpha_m^{{\bphi}}(s))^{-1}+ F_{m}^{m,m}(s)\right]^{-1}F_{m}^{1,m}(s).
\]
}

\section{Numerical Examples}\label{sec:NumEx}
In this section, we apply our approach to three model problems in unbounded domains: 2D diffusion, the 3D quasistationary Maxwell's system, and the 2D wave equation.

 A critical requirement for our method is that the matrix $A$ is s.p.d. and accurately approximates the resolvent for $s$ away from the negative real semi-axis. This allows us to use the block Lanczos algorithms without relying on absorbing boundary conditions or perfectly matched layers (PML), which would lead to non-Hermitian matrices.
 
Instead, we use the non-absorbing optimal grid with Dirichlet boundary conditions proposed in \cite{idkGrids}. This grid follows a log-centered geometric progression with a progression factor of $\exp(\pi/\sqrt{N_{\rm opt}})$, where $N_{\rm opt}$ denotes the number of grid points in the exterior domain.

Among all geometric progression grids, this choice yields the optimal convergence rate of $\exp(-\pi\sqrt{N_{\rm opt}})$ for the spectral measure over the optimization interval.\footnote{Superior rates can be achieved via Zolotarev or Neumann approximations \cite{idkGrids}. Additionally, approximations more closely linked to interior finite-difference discretizations are proposed in \cite{druskin2016near}, where they are also applied to the discretization of PML.}

The generic configuration combining a uniform interior grid with this optimal exterior grid is illustrated in Figure~\ref{fig:Heat2config}. Note that the figure depicts only the first steps of the exponentially expanding exterior grid, as the step sizes grow rapidly thereafter.

\subsection{SISO 2D Diffusion test case}\label{sub:2d}

We discretize the symmetrized elliptic operator  
\be\label{eq:2d}
A\approx -\sigma({\bf x})^{-\frac{1}{2}}\Delta\sigma({\bf x})^{-\frac{1}{2}}
\ee
on an unbounded domain using second-order finite differences over a $300\times 300$ grid with $h=1$.  To approximate this, we consider a Dirichlet problem on a bounded domain with $N_{\rm opt} = 10$ geometrically increasing grid steps in the exterior and a uniform grid in the interior part. This formulation is relevant to both heat transfer and the scalar 2D diffusion of electromagnetic fields, as well as to the wave propagation problem described in subsection~\ref{sub:wave}. 
{\clr
The  function $\sigma(\bx)$ is illustrated in Figure~\ref{fig:Heat2config}.  Here, $B$ is represented as a discrete delta function vector with a single nonzero entry at the grid node corresponding to the transducer location, shown as $\times$. The configuration is equivalent to the one presented in \cite{zimmerling2025monotonicity}.}

{\clr Figure~\ref{fig:SISO_07_2026} presents the results for this model with a single vector $B$ ($p=1$). A very brief, nearly unnoticeable initial sublinear convergence phase (due to the presence of well-separated eigenvalues near the boundary of the spectral interval) is followed by a superlinear and then a linear convergence stage. Initially, the error is dominated by the approximation of the well-separated part of the spectrum, and the conventional Gau{\ss} approach is superior.} However, both, the Gau{\ss}--Radau~(GR) and the Kreĭn--Nudelman~(KN) formulations become competitive for the interval of {\clr linear} convergence, where the error is dominated by the approximation of the almost continuous spectrum of matrix $A$. The arithmetic average of the Gau{\ss} and Gau{\ss}--Radau quadrature (GR avg.) is the average of two S-fractions and has a similar convergence rate. Thanks to its adaptive properties, the Kreĭn--Nudelman method (a single S-fractions) shows generally better convergence than the averaged Gau{\ss} and Gau{\ss}--Radau, however, the latter is much smoother.  In the future we  plan to make the optimization routine used to find $\phi$ in the Kreĭn--Nudelman approach more robust since the oscillations observed in the convergence curve originate from a lack of robustness of the current optimization framework. A possibility of such an improvement follows from Figure~\ref{fig:SISO_07_2026}-a, where the optimal $\phi$ found by true error minimization  (which we call ``cheating'' because the true solution is required) compares with the one found in our energy optimization framework. Averaging the $\phi$ obtained for several iterations will lead to a better-behaved convergence, and the averaged $\phi$ is close to the "cheated" curve. 
\begin{figure}[!ht]
	\centering
	\includegraphics[trim=0mm 0mm 0mm 0mm,clip, width = 0.55 \linewidth]{\pathfig/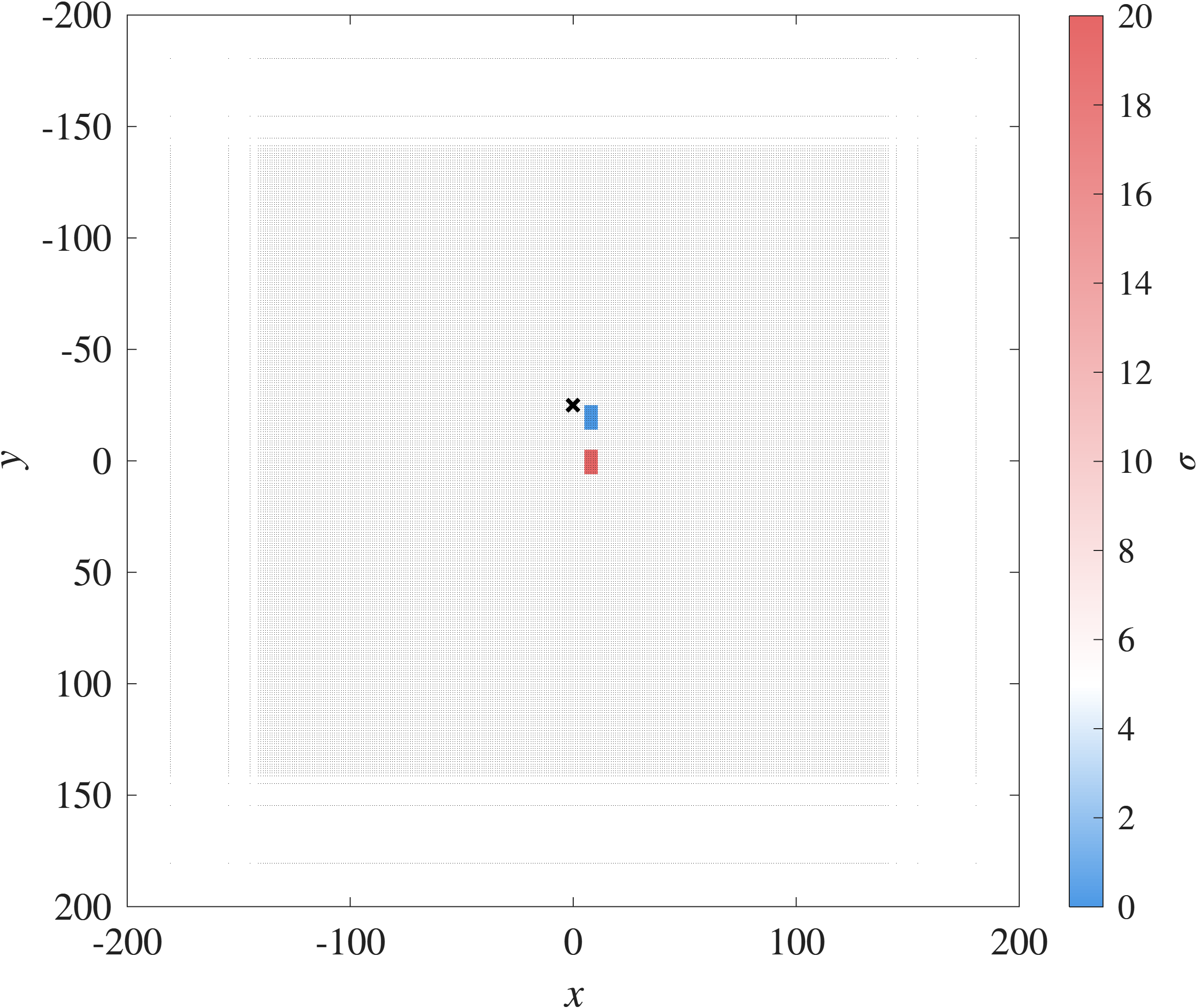}

	\caption{Grid, heat conductivity $\sigma(\bx)$ and transducer location ($\boldsymbol \times$) of the heat diffusion {\clr test case}. In this SISO case only one source $p=1$ is used. The primary grid points are shown as small black dots. Only the first geometrically increasing grid steps are shown. {\clr This configuration is identical to the first example in \cite{zimmerling2025monotonicity}.}}\label{fig:Heat2config}
\end{figure}

\begin{figure}[h!]
 \centering

 \begin{subfigure}[b]{.94\linewidth}
 \centering
  \includegraphics[width = 0.5\linewidth]{\pathfig/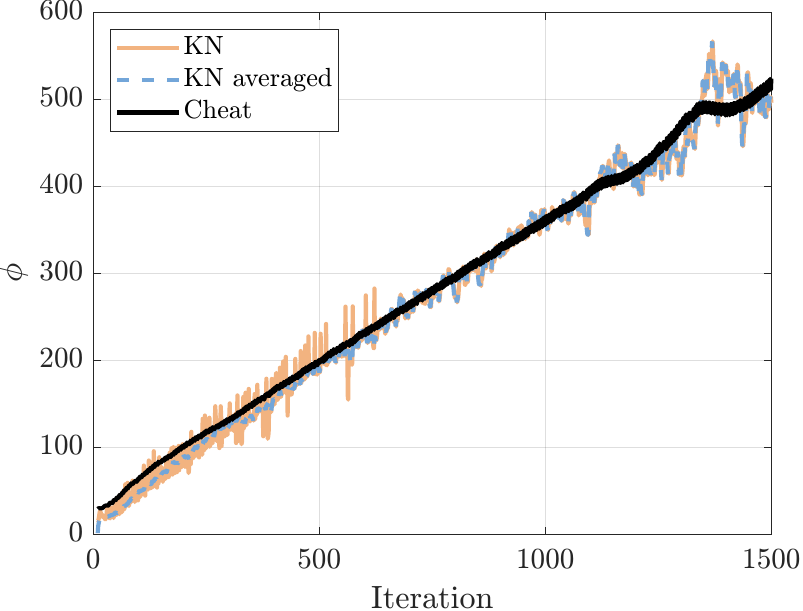}
 \caption{Recovered values for $\phi$ from dissipative energy maximization compared to the value of $\phi$ that minimizes the distance to the true ${\cal F }(s)$ for $s\in[10^{-5}\,\, 10^0] \cup \imath[10^{-5}\,\, 10^0]$. We call this the cheated $\phi$. }
 \label{fig:4_config}
 \end{subfigure}%
 \hspace{1em}

 \begin{subfigure}[b]{.47\linewidth}
 \centering
 \includegraphics[width = 1\linewidth]{\pathfig/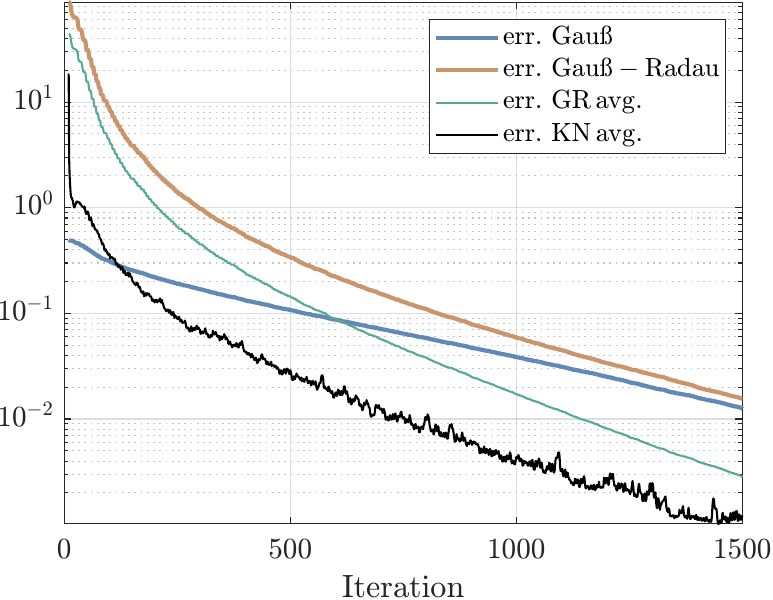}
 \caption{Convergence for $s=4\cdot 10^{-5}\imath$.}
 \label{fig:EMImag}
 \end{subfigure}%
 \hspace{1em}
 \begin{subfigure}[b]{.47\linewidth}
 \centering

\includegraphics[width = \linewidth]{\pathfig/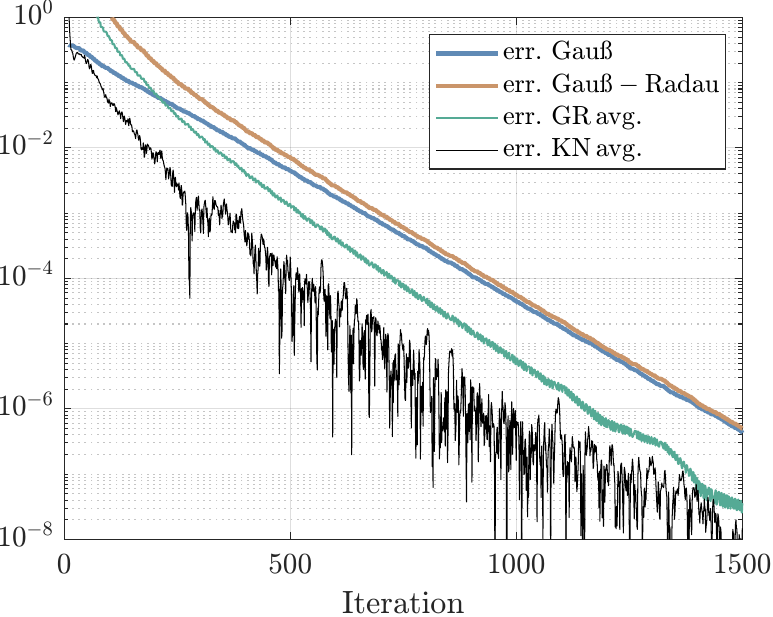}
 \caption{Convergence for a real $s=3\cdot10^{-4}$ value. }
 \label{fig:EMReal}
 \end{subfigure}
\caption{Converge results for the 2D diffusion case in the SISO setting ($p=1$).}
\label{fig:SISO_07_2026}
\end{figure}

\subsection{MIMO 3D Electromagnetic  diffusion \texorpdfstring{\clr test case}{test case}}\label{sec:EMdiff}
In this example, we consider the solution of the 3D Maxwell equations in the quasi-stationary (diffusive) approximation. Krylov subspace methods {\clr for} the computation of the action of matrix functions in this context were first introduced in \cite{druskin1988spectral} and remain a widely used approach in geophysical applications. The diffusive Maxwell equations in $\mathbb{R}^{3}$ can be rewritten as  
\be\label{eq:max}
(\nabla \times \nabla \times +\sigma(\bx) \mu_0 s)\mathcal{E}_{(r)}(\bx,s) = - { s}\mathcal{J}^{\rm ext}_{(r)}(\bx,s), \quad \bx\in \Omega,
\ee  
where $\sigma$ represents the electrical conductivity, $\mu_0$ is the vacuum permeability, and $\mathcal{J}^{\rm ext}_{(r)}$ is the external current density associated with the transmitter index $(r)$, generating the electric field $\mathcal{E}_{(r)}(\bx,s)$. For $s\in\CC\setminus (-\infty, 0)$ , the solution of \eqref{eq:max} vanishes at infinity,  
\be\label{eq:inftymax}  
\lim_{\bx\to\infty} \|\mathcal{E}_{(r)}(\bx,s)\|=0.  
\ee  

To approximate the Maxwell equations \eqref{eq:max}-\eqref{eq:inftymax}, we employ the conservative Yee grid, following the methodology outlined in \cite{druskin1988spectral}. Similarly to the approach used for the diffusion problem in $\mathbb{R}^2$, we truncate the computational domain by imposing Dirichlet boundary conditions on the tangential components of $\mathcal{E}_{(r)}$ at the boundary. The grid is constructed with an optimally spaced geometric progression, ensuring spectral convergence of { condition~\eqref{eq:inftymax} at infinity}. In our setup, we use $N_{opt}=6$ geometrically increasing grid steps with a scaling factor of $\exp{(\pi/\sqrt{6})}$, as proposed in \cite{idkGrids}. The resulting grid consists of $N_x=80$, $N_y=100$, and $N_z=120$ points, yielding a matrix $A$ of size $N=2\,821\,100$.  

{\clr The conductivity $\sigma$ is homogeneous and set to $10^{1}$ throughout the domain, except for two inclusions where $\sigma=10^{-1}$}, as shown in Figure~\ref{fig:5_config}. The excitation consists of six magnetic dipoles (approximated by electric loops): three positioned above the inclusions and three located at their center, as indicated by the arrows in Figure~\ref{fig:5_config}. This setup is analogous to the tri-axial borehole tools used in electromagnetic exploration geophysics \cite{saputra2024adaptive,zimmerling2025targeted}, leading to a matrix $B$ with six columns.  

 Similarly to the results for the heat operator, the 3D results show an advantage of the Kreĭn--Nudelman method and averaged Gau{\ss} and Gau{\ss}--Radau quadrature rules  over stand-alone  Gau{\ss}  and  Gau{\ss}--Radau on the interval of {\clr linear} convergence. On this interval the error decays linearly due to the dense part of the spectral distribution. Overall, the Kreĭn--Nudelman method  shows the best convergence, and due to the density of $A$'s spectrum for the 3D problem the convergence is much smoother than in the 2D case.

In Figures~\ref{fig:MIMOEMSweepReal} and~\ref{fig:MIMOEMSweepImag}, the errors after
$m=730$ block iterations are shown for a wide range of real and imaginary shifts,
respectively. The Kre\u{\i}n--Nudelman method performs best over the displayed
frequency intervals. For a representative frequency in the middle of this
interval, Figures~\ref{fig:MIMOEMReal} and~\ref{fig:MIMOEMImag} show the
convergence behavior as the number of iterations increases.

 The convergence behavior for the 3D  case is more regular in our experience than in the 2D case, and ${\bphi}$ varies smoothly with the iterations, so averaging over iterations is not necessary. The discrete 3D operator has a denser spectral distribution, so the functional in the objective~\eqref{eq:Optimize} is less affected by  individual Ritz values which can be rather irregular due to instability of the simple block Lanczos algorithm in computer arithmetic. For simplicity, we optimized a scalar truncating factor $\bphi = \phi I_p$ for this MIMO case, as more complex parametrizations did not show immediate improvements.

\begin{figure}[h!]
 \centering

 \includegraphics[width = 0.4\linewidth]{\pathfig/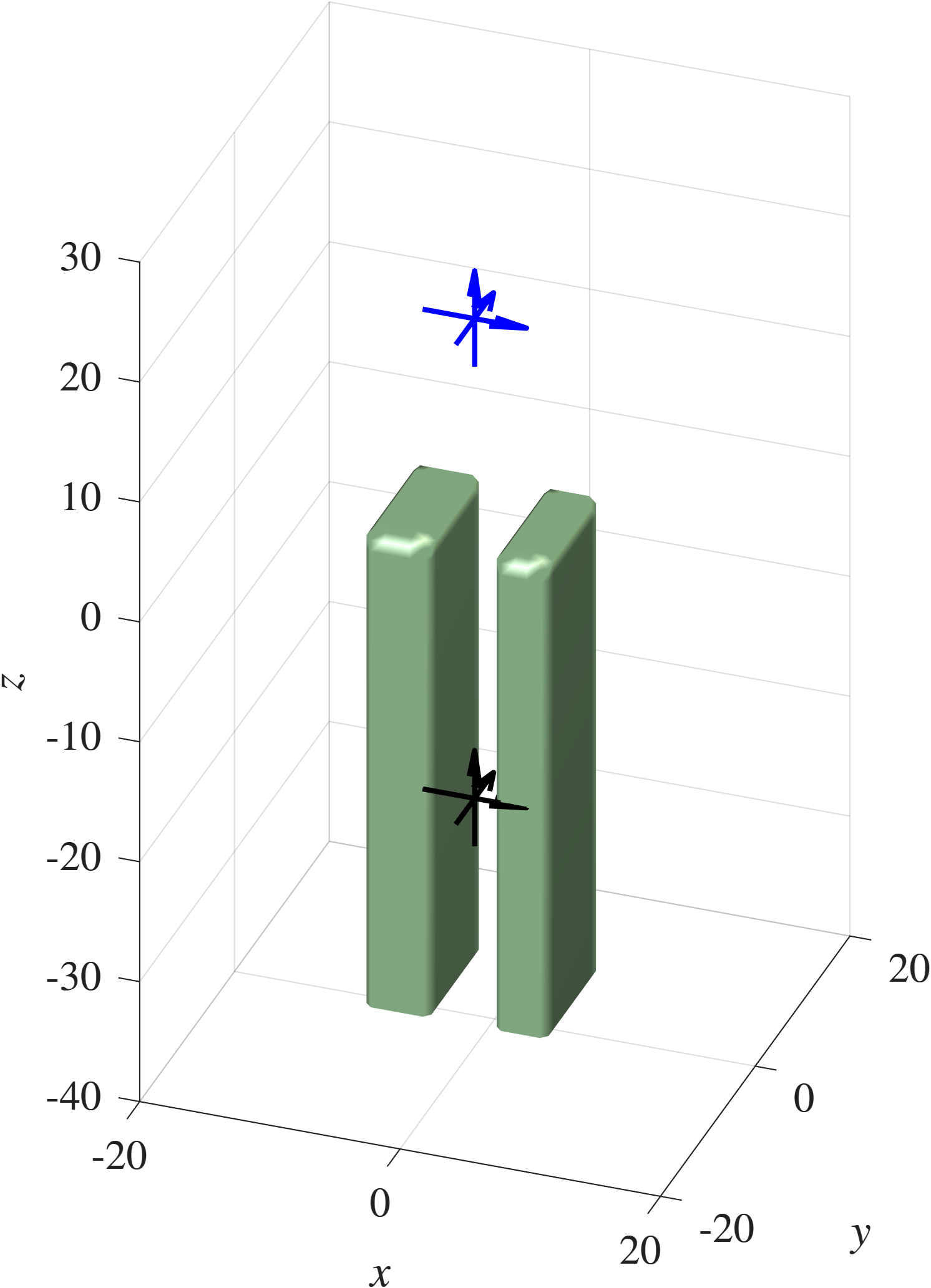}
 \caption{3D rendering of the simulated configuration with two inclusions for the tri-axial electromagnetic borehole tool, $p=6$. {\clr This configuration is identical to the third example in \cite{zimmerling2025monotonicity}.}}
 \label{fig:5_config}
\end{figure}

\begin{figure}[h!]
 \centering
 \begin{subfigure}[b]{.47\linewidth}
 \centering
\includegraphics[width = \linewidth]{\pathfig/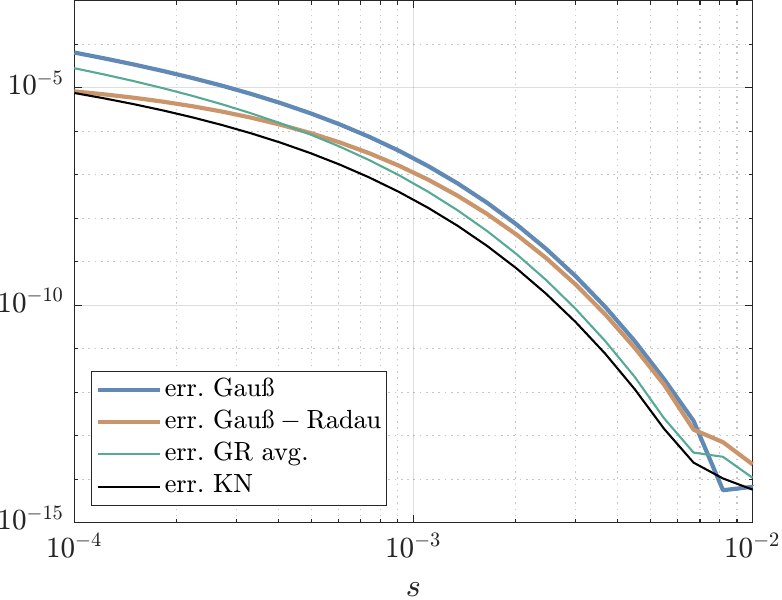}
 \caption{Error norms for various {\clr real} values of $s$ after $m=730$ iterations.}
 \label{fig:MIMOEMSweepReal}
 \end{subfigure}
 \hspace{1em}
 \begin{subfigure}[b]{.47\linewidth}
 \centering
 \includegraphics[width = 1\linewidth]{\pathfig/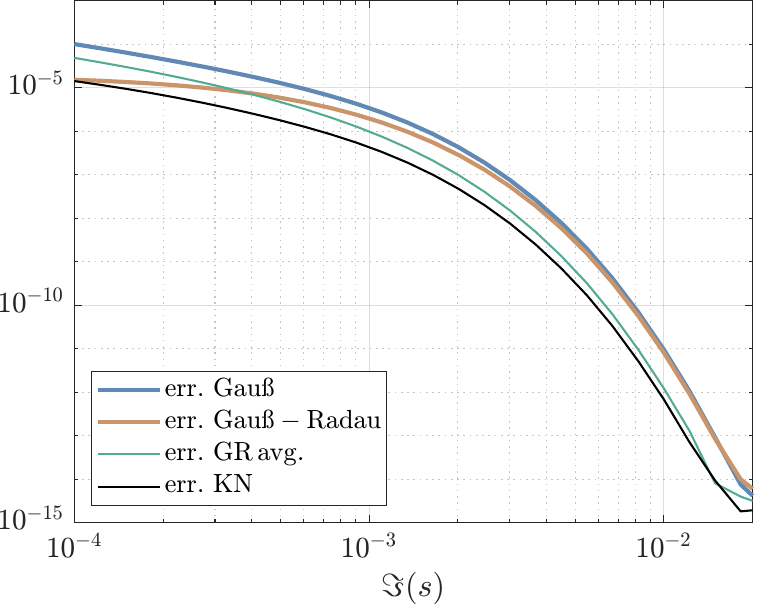}
 \caption{Error norms for various imaginary values of $s$ after $m=730$ iterations.}
 \label{fig:MIMOEMSweepImag}
 \end{subfigure}%

 \begin{subfigure}[b]{.47\linewidth}
 \centering
\includegraphics[width = \linewidth]{\pathfig/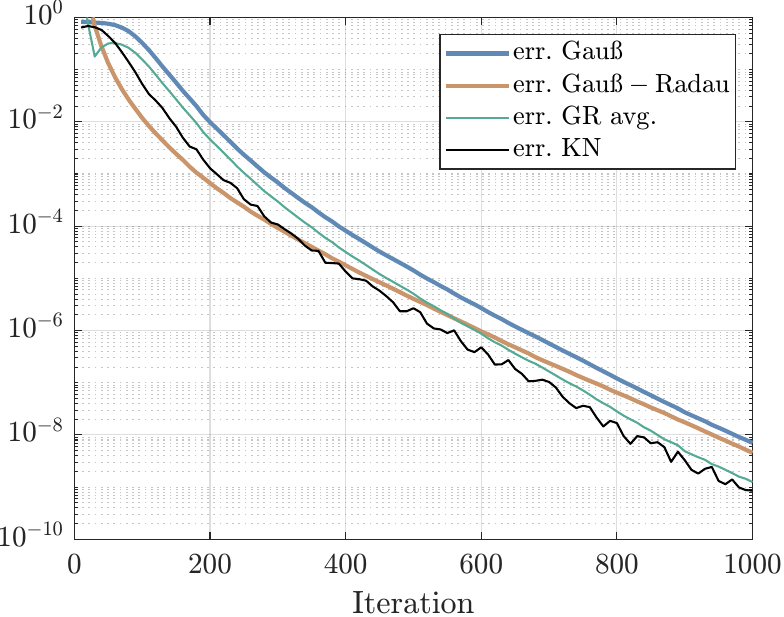}
 \caption{Convergence for $s=10^{-3}$.}
 \label{fig:MIMOEMReal}
 \end{subfigure}
 \hspace{1em}
 \begin{subfigure}[b]{.47\linewidth}
 \centering
\includegraphics[width = \linewidth]{\pathfig/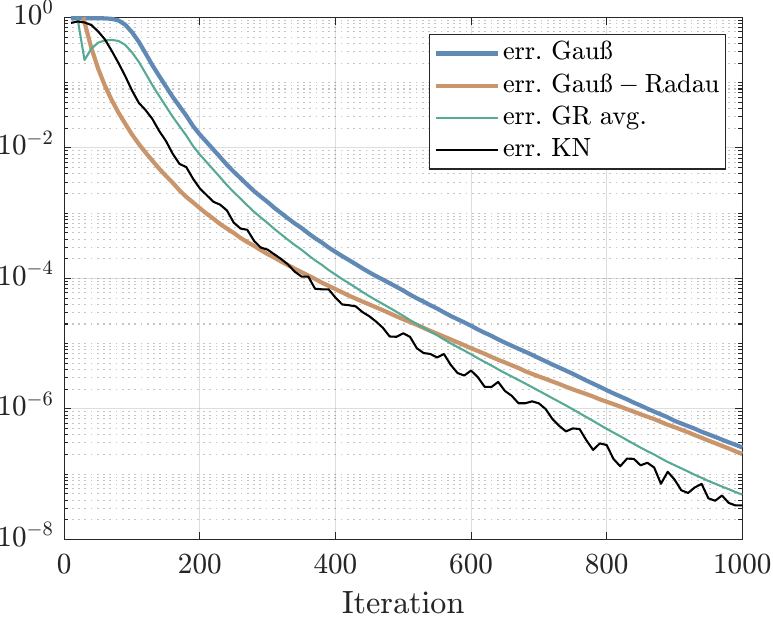}
 \caption{Convergence for  $s=10^{-3} \imath$.}
 \label{fig:MIMOEMImag}
 \end{subfigure}
\caption{Converge results for the 3D EM case in the MIMO setting ($p=6$). As before ``GR average'' is arithmetic average of  Gau{\ss} and Gau{\ss}--Radau quadratures.}
\label{fig:MIMO_03d}
\end{figure}

\subsection{Damped wave-propagation in unbounded domains via Kreĭn--Nudelman extension}\label{sub:wave}

 Let us consider the SISO variant of the 3D Maxwell's problem from subsection~\ref{sec:EMdiff} in the propagative (wave) regime. The solution for this case can be obtained with  $s=(\imath\omega+\epsilon)^2$, $\epsilon>0$  when $\omega > \epsilon$, i.e., the propagative effects are dominant, and still $\epsilon$ is large enough that the wave decays well in the exterior optimal grid (or for practical purposes the grid should be chosen large enough for the spurious reflections to be sufficiently small). This problem corresponds to electromagnetic propagation with harmonic frequency $\omega$ in a lossy medium with dielectric permittivity $(1-\frac{\epsilon^2}{\omega^2})\sigma$ and electrical conductivity   $2\epsilon\sigma$. We present the convergence results of such a simulation in Figure~\ref{fig:decaySolution}. Compared to the cases presented in the previous subsection, the discreteness of the Lanczos spectral approximations has a stronger effect here and the Kreĭn--Nudelman method significantly out-performs the discrete spectral approximations (Gau\ss, Gau{\ss}--Radau and their average)  until the error reaches the level of the spurious reflected wave, which is of order  $10^{-2}$  relative to the exact solution. Then the discreteness of the Lanczos approximation and the exact spectral distribution become of the same level, and the averaged Kreĭn--Nudelman  method's convergence slope becomes comparable to that of the other methods. 
 
\begin{figure}[ht]
\centering
\includegraphics[width=0.45\linewidth]{\pathfig/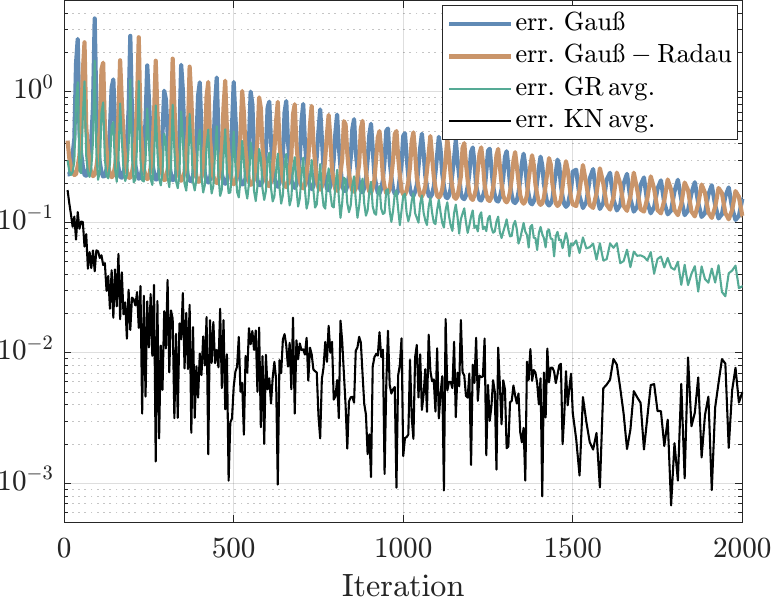}
\caption{Convergence in the wave propagation regime for the 3D electromagnetic problem (subsection~\ref{sec:EMdiff}) with $p=1$. The configuration uses only the central $x$-dipole from Figure~\ref{fig:5_config}, with $s=-0.1 + 0.001\imath$ ($\omega \approx 200 \epsilon$). The legend is the same  as in Figure~\ref{fig:MIMO_03d}. The Kreĭn--Nudelman extension shows superior convergence before stagnating (in fact approaching the convergence speed of the other quadratures) at the level of spurious reflections. {\clr Here, the stronger improvement (compared to the ones in Figures~\ref{fig:MIMO_03d} for the same model but diffusive regime)   during the initial iterations is because the Gau{\ss} and Gau{\ss}--Radau approximations near the branch cut are particularly sensitive to inaccuracies in slowly converging Ritz values, while the Kre\u{\i}n--Nudelman approximation is much less sensitive to that due to the fact that the damped Ritz values lay in the second Riemann sheet, and moreover their distance to the branch cut is proportional to their misfit -- the last component of the corresponding eigenvector of $T_m$. }  }
\label{fig:decaySolution}
\end{figure}

To illustrate the qualitative advantage of the Kre\u{\i}n--Nudelman approximation,  we consider experiments with the state solution $(A+sI)^{-1}B$ for the 2D wave problem (not transfer functions as in the previous examples) using the same setup as in subsection~\ref{sub:2d}, i.e., the same s.p.d.  discretization $A$ as in \eqref{eq:2d} and $B$ in the SISO numerical example of subsection~\ref{sub:2d}. We choose $s\in\CC$  close to the negative real semiaxis ($\omega \gg \epsilon$) to show the effect of boundary reflections. We  compute $\Re \bQ_m (T_m + (\imath\omega+\epsilon)^2 I )^{-1}B \cdot \exp{(\imath \omega t )}$, $\Re \bQ_m (\tilde T_m + (\imath\omega+\epsilon)^2 I )^{-1}B \cdot \exp{(\imath \omega t )}$, and $\Re \bQ_m (\hat T_m^{{\bphi}}(s) + (\imath\omega+\epsilon)^2 I )^{-1}B \cdot \exp{(\imath \omega t )}$ for respectively  Gau{\ss},  Gau{\ss}--Radau and Kreĭn--Nudelman quadratures, for $m=1000$ and $m=2000$ in, respectively, Figures~\ref{fig:wave1000} and   \ref{fig:wave2000}.
For comparison we show the exact solution of the discretized problem $\Re(A + (\imath\omega+\epsilon)^2 I )^{-1}B \cdot \exp(\imath \omega t)$ in the rightmost column of both the  figures.
The top rows in these figures show 2D snapshots for a single $t$.

\begin{figure}[ht]
\centering
\includegraphics[width=\linewidth]{\pathfig/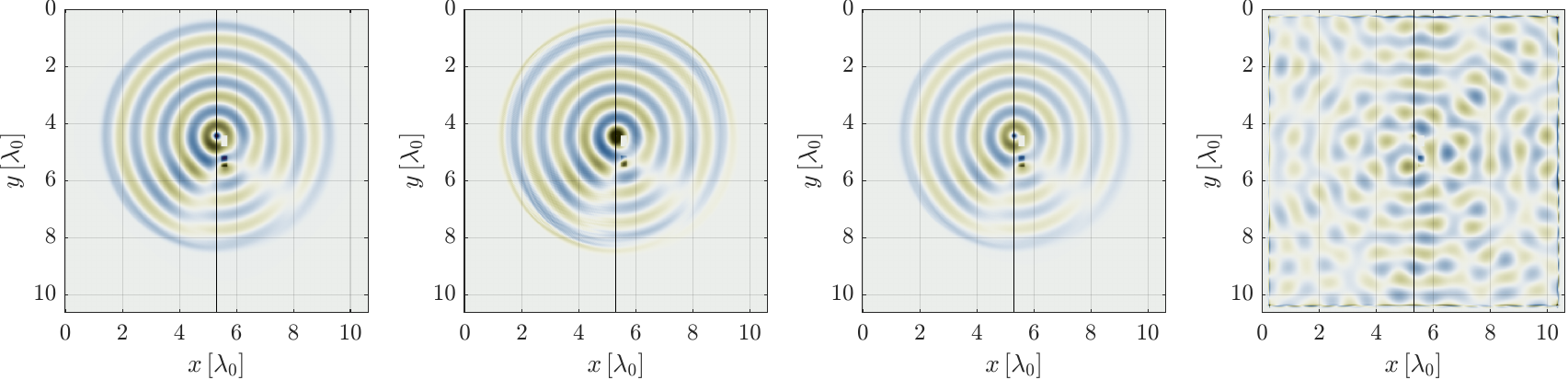}
\includegraphics[width=\linewidth]{\pathfig/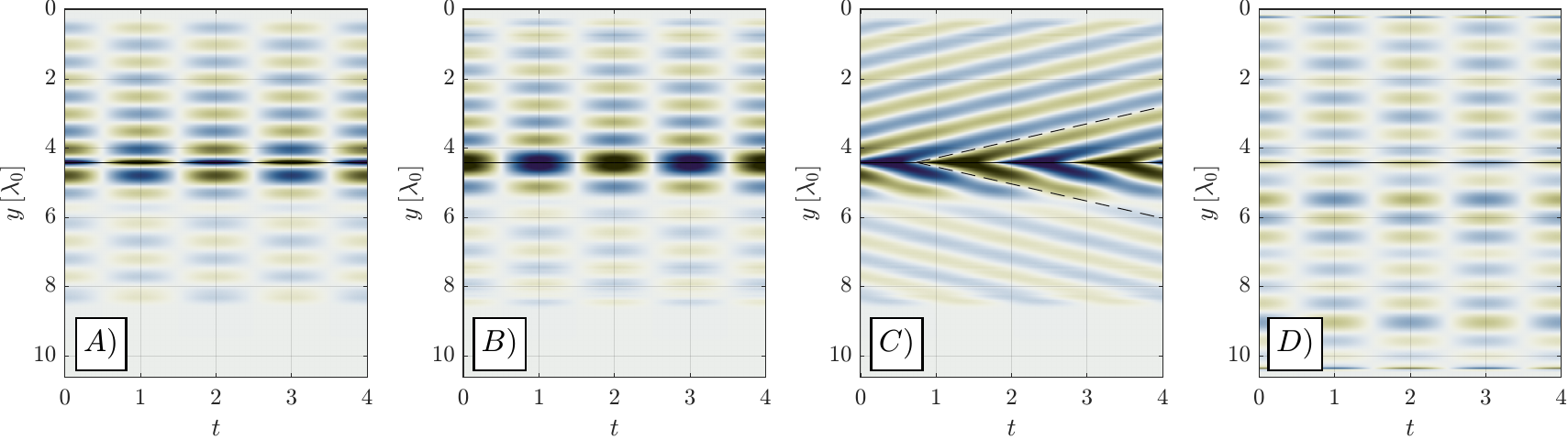}
\caption{Solution approximation of $(A + (\imath\omega+\epsilon)^2 I )^{-1}B$ near the negative real semiaxis after $m=1000$ iterations. \emph{Top Row:} Real part of the state solution for A) the Gau{\ss}, B) Gau{\ss}--Radau, C) Kreĭn--Nudelman approximations, compared with D) the exact solution. \emph{Bottom Row:} Time-domain evolution cross-sections sampled along the vertical black line.  Only the Kreĭn--Nudelman solution (C) exhibits propagating waves traveling along the d'Alembert characteristics (dashed lines). 
In contrast, the Gau{\ss} and Gau{\ss}--Radau approximations yield standing waves consistent with Dirichlet and Neumann boundary conditions, respectively --- reflecting the specific truncation properties of the underlying Stieltjes string.}
\label{fig:wave1000}
\end{figure}
\begin{figure}[h!]
\centering
\includegraphics[width=\linewidth]{\pathfig/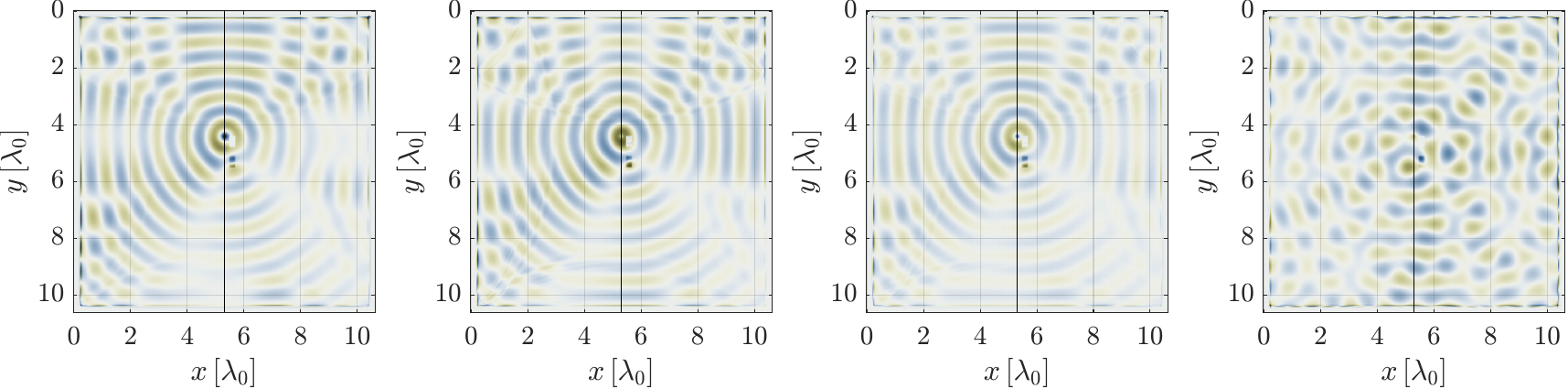}
\includegraphics[width=\linewidth]{\pathfig/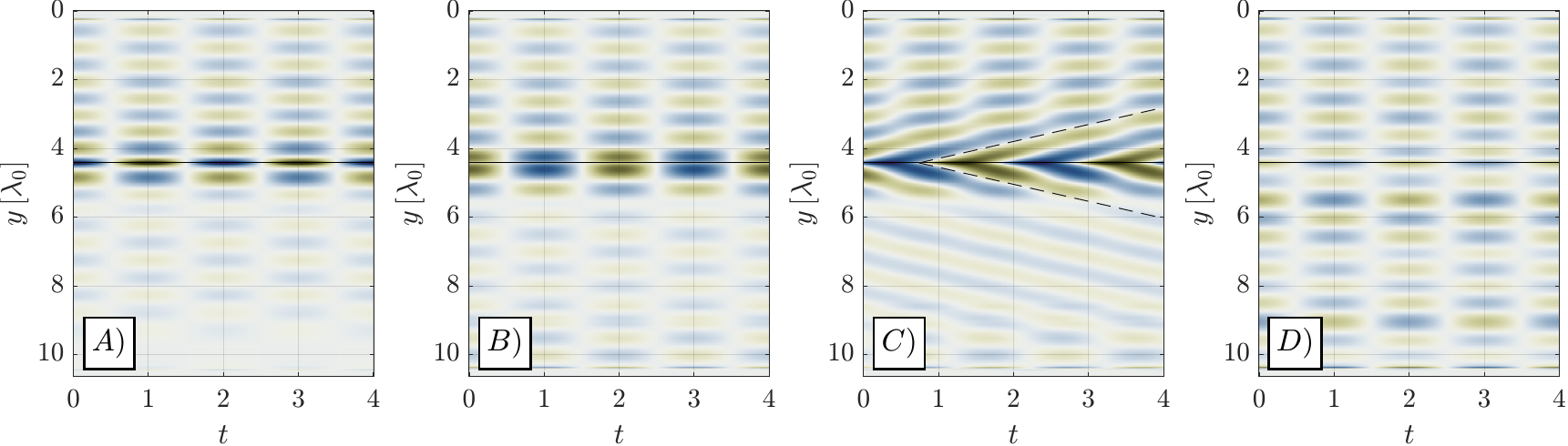}
\caption{Solution approximation for $m=2000$ iterations. Layout and parameters are identical to Figure~\ref{fig:wave1000}. Note that all approximations now show boundary reflections due to converging Ritz values and vectors.} 
\label{fig:wave2000}
\end{figure}

 Let us start with the case $m=1000$. As shown in Figure~\ref{fig:wave1000}, all three approximate solutions exhibit rather regular spherical wavefronts with noticeable scattering effects resulting from the inclusion. The exact solution, however, shows significant irregularity due to waves reflected from the Dirichlet boundary of the computational domain. On the lower row, the time evolution along the vertical cross-section indicated in the top row is shown.
All solutions except the Kreĭn–Nudelman solution exhibit standing-wave behavior, dominated by Ritz vectors or individual eigenmodes. This behavior is expected, since all three dynamics originate from discrete spectra with real eigenfunctions. In contrast, the Kreĭn–Nudelman cross-section clearly captures an outgoing wave propagating along the d’Alembert characteristics of the exact wave equation \[-\sigma({\bf x})^{-\frac{1}{2}}\Delta\sigma({\bf x})^{-\frac{1}{2}}u+u_{tt}=0,\] thereby providing at least a qualitatively accurate approximation of the solution in an open domain, which is not possible with non-absorbing approximations. 

Since the Stieltjes string for Gau{\ss} has a Dirichlet termination, the wave in the top panel of~\ref{fig:wave1000}-A is terminated at a node. Conversely, the Gau{\ss}--Radau approximation leads to a termination at an antinode, a Neumann condition.

For the case  $m=2000$ presented in Figure~\ref{fig:wave2000}, snapshots of all three approximate solutions show increased irregularity. This arises because the discrete spectra have begun to converge to the exact operator's spectrum; consequently, the approximations move closer to the exact solution and further from the desired open-domain behavior.

Due to the convergent part of the discrete spectra, they are closer to the exact solution in the rightmost column and further from the desired outgoing wave behavior. 
The cross-section results in the bottom row for the Gau{\ss} and Gau{\ss}--Radau again show standing waves similar to the case with $m=1000$ (the exact solution is the same for both figures).  The Kreĭn--Nudelman waves for $m=2000$ are still dominantly outgoing --- however, they have more irregularities due to convergent parts of the spectrum, contaminating the solution with standing waves. Eventually, as $m\to\infty$ the Kreĭn--Nudelman (as well as the other two approximate solutions) would converge to the exact solution, so it will again be just a superposition of standing waves. This change is consistent with the dependence of the scattering poles on $m$ shown in Figure~\ref{fig:lists}.

This situation is reminiscent of solving ill-posed linear systems by Krylov subspace iterations: while the fully converged solution can be highly sensitive to data errors, intermediate iterates often provide effective regularized approximations (see, e.g., \cite{Reichel2020SimpleSC}). We therefore can potentially adopt a similar strategy here, treating the  error of the discretization of the exterior domain as an error in the data.
For robust wave-propagation computations using the Kreĭn–Nudelman formulation, the number of iterations must be chosen carefully to avoid over-convergence. In particular, iterations should be terminated once the data misfit $\|\hat\blF-\blF\|$ becomes comparable to the boundary truncation  error in the operator $A$. This requires reliable estimates of both error sources to define an appropriate stopping criterion. Sharp estimates for the latter error using boundary truncation algorithm via so-called optimal grid are available \cite{idkGrids, druskin2016near} convergent grid coarsening.

\section{Conclusions and Outlook}

In this work, we have demonstrated that the approximation of matrix forms can be analyzed by rewriting the underlying Lanczos process as a matricial Stieltjes continued fraction. By interpreting this structure through the lens of block Stieltjes strings, we establish a direct correspondence between the algebraic truncation of the Lanczos recursion and the physical behavior of a vibrating block Stieltjes string. Specifically, we show that a block Lanczos approximation at a finite iterate corresponds to a truncated string, which inherently introduces spurious reflections at the computational boundary.

We find that classical quadrature rules are characterized by specific boundary conditions: Gau\ss\ quadrature induces a Dirichlet-type reflection, whereas Gau\ss--Radau quadrature induces a Neumann-type reflection. Consequently, the averaging of these two rules as studied in \cite{zimmerling2025monotonicity} can be understood as an attempt to cancel these opposing non-absorbing boundary effects, i.e. they correspond to a superposition of a string with an even and odd extension.

To go beyond these classical limits, we proposed an absorbing boundary condition inspired by the Kre\u{\i}n--Nudelman representation of rational Stieltjes functions. This boundary condition manifests as a damping term depending on the spectral parameter at the truncation point of the Stieltjes string, effectively introducing a branch cut in the resolvent. {\clr The damping parameter is selected by a port-Hamiltonian energy balance, in which the objective is the power absorbed at the truncation normalized by the stored energy (Appendix~\ref{sec:port}).} We term this method Kre\u{\i}n--Nudelman method and it effectively suppresses spurious reflections. This approach provides a superior framework for approximating matrices with almost continuous spectral distributions, particularly those arising from the discretization of PDE operators on unbounded domains. {\clr The gain comes from replacing the discrete Ritz approximation of a dense spectrum by a smooth spectral density.}

In the regime of linear convergence, both the Kre\u{\i}n--Nudelman extension and the averaging of Gau\ss/Gau\ss--Radau quadrature rules yield significant error reductions when computing transfer functions for large-scale 2D diffusion and 3D quasi-stationary Maxwell systems. While the Kre\u{\i}n--Nudelman approach consistently outperforms quadrature averaging in terms of absolute error, it exhibits a less regular convergence profile. We hypothesize that this loss of monotonicity may be mitigated by refining the dissipation maximization procedure that we use to determine the optimal Kre\u{\i}n--Nudelman damping parameter.

For wave propagation problems in unbounded domains, the Kre\u{\i}n--Nudelman extension provides the strongest acceleration and also a qualitative correction to the solution by transforming nonphysical standing waves into the expected outgoing waves. However, this phenomenon can suffer from over-convergence as Ritz values converge, and finding an optimal stopping criterion consistent with the discretization error is required to make it robust. 

{\clr 

Looking forward, the most direct extension is to a more general class of
terminators used in quantum physics to accelerate Lanczos recursions for Green's
function and spectral density calculations with continuum spectra, e.g.,
see~\cite{HaydockHeineKelly1972,Haydock1980Recursion,BeerPettifor1984Terminator,BeerPettifor1984NATOASI,LuchiniNex1987Stitching,PinnaLuntvonKeyserlingk2025Stitching}.
These assume that the Lanczos recursion converges to a stationary limit, so that
asymptotically it becomes a shifted Chebyshev recursion whose spectral function
is a Stieltjes continued fraction (S-fraction) with convergent coefficients.
Such an S-fraction can be truncated exactly by a function of the form
$b\sqrt{s(s-a)}$, e.g.\ the infinite S-fraction with coefficients $\gamma$ and
$s\hat\gamma$ gives
\be\label{eq:sqrt}
\frac{\gamma}{2} + \frac{1}{s\hat{\gamma} + \frac{1}{\gamma + \frac{1}{s\hat{\gamma} + \scalebox{0.5}{$\ddots$}}}} =  \frac{\gamma}{2s} \sqrt{s \left( s + \frac{4}{\gamma\hat{\gamma}} \right)},
\ee
hence the name square-root terminator, or recursion method. The
Kre\u{\i}n--Nudelman extension can be seen as a particular case of such a terminator. Unlike the square-root terminator, however, a
full square-root terminator can in principle resolve scattering poles on the
second Riemann sheet rather than only smoothing the spectral density, which is
what makes it attractive as a next step.

The recursion method was introduced in \cite{HaydockHeineKelly1972} to study the
electronic band structure of materials, and the square-root terminator
in~\cite{BeerPettifor1984Terminator}; see \cite{PettiforWeaire1985} for a
survey. In quantum-mechanical computations the coefficients of a Schr\"odinger
potential are often available analytically, so the Lanczos recursion can be
carried out in closed form but with exponentially growing complexity; attaching
a square-root terminator to the last computed fraction let the wave function
propagate out of the truncated region into an effective infinite medium, which
made otherwise intractable problems solvable. The approach has nonetheless seen
limited use in computational linear algebra, since the assumption of a
stationary recursion limit is restrictive and breaks down on intervals of
nonuniform spectral distribution, an issue compounded by the instability of the
simple Lanczos algorithm without reorthogonalization. We hope that our
damping optimization can supply the missing parameter selection in this more
general setting.

{\clr Two further alternatives are worth mentioning.} The ``stitching'' method
preconditions with the tail of a reference Lanczos recursion
\cite{PinnaLuntvonKeyserlingk2025Stitching} and is related to ours through the
Stieltjes moments it preserves; while it provides algebraic acceleration in
quantum scattering, preliminary estimates for wave scattering with compactly
supported perturbations suggest the potential for exponential or even
superlinear convergence. A different route is to represent the square-root
branch cut explicitly through the ``quadratic'' Pad\'e method
\cite{MayerNuttallTong1984QuadraticPade, MayerTong1985QuadraticPade,
Fasondini2019QuadraticPade}, itself a special case of the Hermite--Pad\'e method
\cite{Aptekarev, Suetin}. Whether either can be realized as a low-rank modification of a Lanczos recursion is not clear, and is itself an interesting question.

In this paper we focused on matrices $A$ obtained
from discretizations of PDE operators $A_\infty$ with continuous spectrum, and
treated the approximation of $A_\infty$ itself only qualitatively. Spectral
approximation of infinite-dimensional operators through convergent sequences of
discretizations has been studied by many authors, e.g., see
\cite{Kuijlaars2000, ColbrookHansen2023, BeckermannGuttel2010} and references
therein. For the exterior PDE problems considered here the task may in fact be
easier: the accuracy of the interior discretization need not change, only that
of the domain truncation, by taking as $A_\infty$ the operator whose grid
extends to infinity. That limit can be approximated with exponentially decaying
error and sharp bounds via optimal grids, or finite-difference Gau{\ss}ian
quadratures \cite{idkGrids, druskin2016near}, and applying such methods
rigorously to convergent sequences of matrices is research in progress.

Beyond the terminator itself, the framework invites generalization to matrix
forms $B^T f(A) B$ for wider classes of $f$, such as the matrix exponential or
Stieltjes functions, which could improve on existing estimates based on
Chebyshev series \cite{druskin1989two}. As with the original square-root
terminator, a natural application is the estimation of spectral densities for
large matrices
\cite{UbaruChenSaad2017SLQ,ColbrookHorningTownsend2021SIAMRevSpectralMeasures,YiMassattHorningLuskinPixleyKaye2025DeltaChebyshev},
where an intrinsically continuous representation of the spectral measure has a
clear advantage over discrete approximations. We likewise anticipate that the
approach extends to other problems exhibiting linear convergence, such as
transfer functions of large-scale graph Laplacians
\cite{zimmerling2025monotonicity}. Randomized enrichment of the initial block is
known to reduce the number of block Lanczos iterations
\cite{ChenEpperlyMeyerMuscoRao2026,PerssonChenMusco2025}, in particular combined
with Gau{\ss} and Gau{\ss}--Radau averaging \cite{zimmerling2025monotonicity}, so
we expect the more strongly adaptive acceleration proposed here to yield larger
reductions still, and combining it with randomized trace estimators
\cite{AvronToledo2011RandomizedTrace,PerssonChenMusco2025} is another direction
for future work.

}


\bmhead{Acknowledgements}

Vladimir Druskin was partially supported by AFOSR
grants FA 9550-20-1-0079, FA9550-23-1-0220,  NSF
grant DMS-2110773 and Texas at Austin’s Deep Imaging Sub-Consortium (of the Research Consortium on Formation Evaluation),
currently sponsored by AkerBP, Baker-Hughes, bp, ENI, Equinor ASA, Halliburton, and Petrobras.  The computations were enabled
by resources in project UPPMAX 2025/2-271 provided by Uppsala University at UPPMAX. Visualizations were created using scientific colormaps by Fabio Crameri~\cite{Crameri2023ScientificColormapsPackage}.

\section*{Ethics declarations}
\bmhead{Conflict of interest}
The authors declare that they have no conflict of interest regarding the publication of this paper.

\section*{Data Availability}
Data sharing is not applicable to this article as no datasets were generated or analyzed during the current study.

\begin{appendices}

\section{Extraction of the Stieltjes parameters from the Lanczos block tridiagonal matrix }\label{ap:A}

This appendix is included for completeness, for derivation, see \cite{ZaslavskySfraction,zimmerling2025monotonicity}. From the construction of $\Alpha_i$'s and $\Beta_i>0$ in Algorithm~\ref{alg:blockLanc} we obtain that $\hKappa_i$, $\gam_{i}$ are full rank, as long as no deflation occurs in the block Lanczos recurrence, as $\Beta_i^T\Beta_i>0$ for $i=2,\ldots,m$, i.e. nonsingularity of all coefficients $\Alpha_i,\Beta_i$ ensures that 
Algorithm~\ref{alg:GammaExtraction} runs without breakdowns. { For some scalings of problems, the extraction can be unstable if $\hgam_i$ or $\gam_i$ run off to zero and infinity, respectively, even though their product is stable. In such cases, it is possible to extract the parameters from $\delta\Alpha_i$ and $\Beta_i$ by imposing that $T_i-E_i\delta\Alpha_i E_i^T$ has a $ p$-dimensional nullspace. From this condition, $\delta \Alpha_i$ can be obtained using a $(i-1)\times p$ dimensional block tridiagonal solve with $p$ right-hand sides, by eliminating the last block row.  }
 
 \begin{center}
\begin{minipage}{.55\linewidth}
 \begin{algorithm}[H]
 	\caption{Extraction algorithm $\gam/\hKappa$:\\ (Block $LDL^T$ Cholesky factorization of block tridiagonal $T_m$)}\label{alg:GammaExtraction}
 	\begin{algorithmic}
 	\normalsize
 		\State Given $\Alpha_i,\Beta_i$ and $\hKappa_1=I_p$ (since $\Beta_1=I_p$)
 		\State $\gam_1^{-1}=\hKappa_{1}^T \Alpha_1 \hKappa_{1}$ 		
		 \For{$i= 2,\dots, m$} 
 		\State $\hKappa_{i}^{-1}\quad	={ -} \gam_{i-1}(\hKappa_{i-1})^T\Beta_i^T$\Comment{$(*)$}
 		\State $\gam_i^{-1}\quad=\phantom{-} (\hKappa_{i})^T \Alpha_i \hKappa_{i}- \gam_{i-1}^{-1}$\Comment{$(\dagger)$} 
 		\EndFor 
  		
 	\end{algorithmic}
 	\label{alg:ExtractGam}
 \end{algorithm}
 \end{minipage}
 \end{center}
 Derivations of different variants of the Algorithm~\ref{alg:GammaExtraction} are given in \cite{zimmerling2025monotonicity,Druskin2016}.
 We should point out that after $m$ steps of the block Algorithm~\ref{alg:blockLanc}
 one also can obtain $\Beta_{m+1}$ without additional mutiplication by $A$, which allows us also to recover $\hKappa_{m+1}$ and $\hgam_{m+1}$.
 \section{Continued fraction interpretation}\label{ap:ContFrac}
 \subsection{Gau{\ss}ian quadrature}
 \begin{proposition}
 The Gau{\ss}ian quadrature can be written as a truncated matricial Stieltjes continued fraction (S-fraction) :
\be\label{eq:S-fraction1}
	\blF_m(s)= \cfrac{1}{s\hgam_1+ \cfrac{1}{\gam_1 + \cfrac{1}{s\hgam_2 +\cfrac{1}{ \ddots \cfrac{1}{s \hgam_m + \cfrac{1}{ \gam_m}} }}}} ,
\ee
where the basic building block of these S-fractions is given by the backward recursion
\be\label{eq:DefCFblock}
	\blC_i(s)=\cfrac{1}{s\hgam_i + \cfrac{1}{\gam_i+\blC_{i+1}(s)}}, \quad \blC_{m+1}(s) = 0,
\ee
with $\blC_i(s)\in\CC^{p\times p}$, $i=1,\ldots, m$ being s.p.d. for real positive $s$.

\end{proposition}
The proof for $p=1$ was known in Stieltjes work or even earlier, and its extension for $p\ge 1$ is given in \cite{zimmerling2025monotonicity}. Here, for completeness, we present its simplified variant.

 \begin{proof}
To prove this statement we use induction. At every step of the induction, we increase the continued fraction and solve a block tridiagonal pencil problem that increases by one block each step and coincides with the full pencil $(Z_m+s \widehat{\Gamma}_m)$ at the last step. To facilitate this we define $\blUi{i}{j}$ as a (triangular) family of $p \times p$ real matrices
\[
\{\blUi{i}{j} \in \mathbb{R}^{p \times p} | j=1,\dots,m \,\,, \,\, i=j,\dots,m\},
\]
where $i$ corresponds to the blocks coupled through the three-term recurrence relation defined by the pencil $(Z_m+s \widehat{\Gamma}_m)$ and $j$ corresponds to the steps of the induction where at the $j$-th step we enforce the (block Neumann) boundary condition 
\be\label{eq:BC}
-{\gam^{-1}_{j-1}{(\blUi{j}{j}-\blUi{j-1}{j})}}=I_p.
\ee
This condition is enforced on the last $m-j$ blocks of the three term recurrence relation defined by the $m-j$ last block rows of $(Z_m+s \widehat{\Gamma}_m)$ which can be written out as
\begin{align}
-{\gam^{-1}_{i}}{(\blUi{i+1}{j}-\blUi{i}{j})} +{\gam^{-1}_{i-1}{(\blUi{i}{j}-\blUi{i-1}{j})}} + s\hgam_i \blUi{i}{j} &=0, 	\quad i=j,\dots,m,\,\, j=1,\dots,m \label{eq:REC} \\
&{\quad} \blUi{m+1}{j}=0.
\end{align}
Note that for $j=1$ this corresponds together with \eqref{eq:BC} to the full pencil problem $(Z_m+s \widehat{\Gamma}_m)[(\blUi{1}{1})^T, \dots, (\blUi{m}{1})^T ]^T=E_1$, which yields $\blF_m=\blUi{1}{1}$.\\

Further, we introduce the notation
\be
\blC_j = \blUi{j}{j}
\ee 
which due to condition \eqref{eq:BC} provides the block Neumann-to-Dirichlet map
\be\label{eq:BNTD}
\blC_i [-{\gam^{-1}_{i-1}{(\blUi{i}{j}-\blUi{i-1}{j})}}] = \blUi{i}{j} \quad \forall i\geq j
\ee
since $\blU_{i\geq j}^{j}$ satisfies the same recursion $\forall j$. This holds because the sub-pencil for the tail $i,\dots,m$ is invariant: the variables $\blUi{i}{j}$ on this tail satisfy the same recurrence and boundary conditions as the intrinsic problem defining $\blC_i$, regardless of the upstream driving index $j$.

As shown later during the induction, these maps are positive definite and thus invertible $\forall s\in\mathbb{C}\setminus(-\infty,0)$ and the inverse gives
\begin{align}
-{\gam^{-1}_{i-1}{(\blUi{i}{j}-\blUi{i-1}{j})}}&=\blC_i^{-1} \blUi{i}{j} \quad\forall i\geq j\\
\leftrightarrow \blUi{i-1}{j}&=(\blC_i+\gam_{i-1})\blC_i^{-1} \blUi{i}{j}
\end{align}
{\clr which for $i=j+1$ and assuming $\gam_{j}>0$ and invertibility of $\blC_{j+1}$ gives}
\be\label{eq:ReccursionStart}
\blC_{j+1}^{-1} \blUi{j+1}{j}=(\blC_{j+1}+\gam_{j})^{-1}\blUi{j}{j}=-{\gam^{-1}_{j}}{(\blUi{j+1}{j}-\blUi{j}{j})}.
\ee

With this in place consider the recursion \eqref{eq:REC} for $i=j$, substitute the condition \eqref{eq:BC} and replace the term $-{\gam^{-1}_{j}}{(\blUi{j+1}{j}-\blUi{j}{j})}$ with the expression derived in \eqref{eq:ReccursionStart} to relate $\blC_{j+1}$ to $\blC_{j}$ as
\be
(\blC_{j+1}+\gam_{j})^{-1}\blC_j + s \hgam_j s \blC_j =I_p
\ee
which gives the basic building block of the material S-fraction as
\be\label{eq:SFracRecBlock}
\blC_j(s)=\cfrac{1}{s\hgam_j + \cfrac{1}{\gam_j+\blC_{j+1}(s)}} ,
	\qquad {\rm Re}\{s\}>0.
\ee

We are now ready to perform the induction, where we show that indeed $\blC$'s are invertible and that the upper recursion gives the desired S-fraction for $\blF_m$

\begin{itemize}
		\item \emph{Base Case}: For the case $j=m$ the condition $\blUi{m+1}{m}=0$ implies $\blC_{m+1}=0$ in \eqref{eq:SFracRecBlock}. Alternatively Equation~ \eqref{eq:REC} for $i=j=m$ can be written as 
\be
\gam_{m}^{-1}\blC_m + s \hgam_m \blC_m =I_p
\ee
which starts the S-fraction as
\be
\blC_m(s)=\cfrac{1}{s\hgam_m + \cfrac{1}{\gam_m}} , \qquad {\rm Re}\{s\}>0.
\ee
which is positive definite and thus invertible.
	\item \emph{Induction Step}: 
	Next for $m>j \geq 1$ we assume that $\blC_{j+1}$ is positive definite such that
\[
\blC_j(s)=\cfrac{1}{s\hgam_j + \cfrac{1}{\gam_j+\blC_{j+1}(s)}} 
\]
extends the S-fraction and ensures that $\blC_j$ is positive definite and invertible. The recursion ends with $j=1$ at which point \eqref{eq:BC} and \eqref{eq:REC} coincide with the full pencil.		 
\end{itemize}

\end{proof}
\subsection{Extension to Gau{\ss}--Radau and Kreĭn--Nudelman quadratures}\label{ap:ext}

To obtain the continuous fraction expression for $ {\hat\blF^{{\bphi}}}_m(s)$, we return to the Gau{\ss} quadrature.
Comparing equations~(\ref{eqn:line1}-\ref{eqn:linem}) with \eqref{eq:DefCFblock} we obtained \eqref{eq:BNTD}, the expression
$$
\blC_i^{-1} \blU_i:= - \frac 1 {\gam_{i-1}} (\blU_i-\blU_{i-1}).
$$
In the finite-difference interpretation of the block Stieltjes string $\blC_i$ connects the solution $\blU_i$ with its backward differences, i.e., $\blC_i^{-1}$ can be interpreted as the block Dirichlet-to-Neumann operator.

Condition \eqref{eqn:linem} yields the last condition of \eqref{eq:DefCFblock}. Following \cite{KreinNudelman1973,KreinNudelman1989}, and extending the Dirichlet-to-Neumann interpretation, we replace it with the ``impedance boundary condition'' 
\be\label{eq:KN}
\blC_{m+1}^{-1} := {(\bphi\sqrt{s})}
\ee
thus we obtain
\be\label{eq:S-fraction2}
	\hat \blF_m^{{\bphi}} (s)= \cfrac{1}{s\hgam_1+ \cfrac{1}{\gam_1 + \cfrac{1}{s\hgam_2 +\cfrac{1}{ \ddots \cfrac{1}{s \hgam_m + \cfrac{1}{ \gam_m+\cfrac{1}{\bphi\sqrt{s}}}}}}}} .\ee
Similar to the limiting transitions written in terms of the $\hat T^{{\bphi}}_m(s)$, we can write in terms of the continued fraction
\be\label{eq:lim1}\lim_{{{\bphi}}\to \infty}\hat \blF^{{\bphi}}_m(s)= \blF_m(s)\ee and
\be\label{eq:lim2}
	\lim_{\bphi\to 0}\hat \blF^{{\bphi}}_m(s)= \cfrac{1}{s\hgam_1+ \cfrac{1}{\gam_1 + \cfrac{1}{s\hgam_2 +\cfrac{1}{ \ddots \cfrac{1}{s \hgam_m } }}}} = \blTF_m(s),\ee
 where $\blTF_m(s)$ is the block Gau{\ss}--Radau quadrature as defined in \cite{zimmerling2025monotonicity}.

 \subsection{Proof of Proposition~\ref{prop:main}}\label{sec:ProofProp1}
 Using the equations introduced in this appendix, we can now proceed with the proof of Proposition~\ref{prop:main}.
 \begin{proof}
 To prove the first statement, we first notice that $\sqrt{s}$ is a two-valued Stieltjes function with the branch cut on $\RR_-$. Then $\blC_i$ are M\"obius transforms of $\blC_{i-1}$, so they recursively map the  complex plane to itself and a Stieltjes function to another Stieltjes function. Thus recursively, the branch cut of $\sqrt{s}$ is mapping to the branch cut with the same location.
 
The second statement is proved similarly to equation~\eqref{bound:GR} in \cite{zimmerling2025monotonicity}. We write the counterpart of the recursion \eqref{eq:DefCFblock} for $\hat \blF^{\bphi}_m(s)$ as
\begin{flalign*}
\hat \blF^{\bphi}_m(s)=\blC_{1}(s), \\ \blC_i(s)=\cfrac{1}{s\hgam_i + \cfrac{1}{\gam_i+\blC_{i+1}(s)}}, \quad i=m,m-1,\ldots, 1, \\ \blC_{m+1}(s) = (\bphi\sqrt{s})^{-1}. \end{flalign*}
{By construction $\blC_{m+1}$ is strictly monotonic with respect to s.p.d. ${\bphi}^{-1}$ for $s\in\mathbb{R}^+$, and then $\blC_i$'s are strictly monotonic with respect to $\blC_{i+1}$ for $i=m,m-1,\ldots, 1$, which gives strict monotonicity of $\hat \blF^{\bphi}_m(s)$ with respect to ${\bphi}^{-1}$. The Gau{\ss} quadrature is obtained in \eqref{eq:lim1} as the limit for ${\bphi}^{-1}\to 0$ so it gives the lower bound, and \eqref{eq:lim2} as the limit for $\bphi^{-1}\to \infty$, so we obtain the Gau{\ss}--Radau quadrature as the upper bound.}
\end{proof}

\begin{remark}
{To extend the first statement of Proposition~\ref{prop:main} to $p>1$ one needs to extend the M\"obius transform to matricial M\"obius transforms (a.k.a. Linear Fractional Transformations (LFTs)). Then one needs to work with matrix-complex planes, which leads to a similar reasoning; however we leave a complete proof for future work.

 The extension of the second statement of Proposition~\ref{prop:main} to $p>1$ can be obtained by applying the monotonicity reasoning to the matrix-valued argument. {Let} $\blF_m(s,\blC_{m+1})=\blC_{1}(s)$ {be the lower recursion truncated by $\blC_{m+1}$}
\begin{eqnarray*}
 \blC_i(s)=\cfrac{1}{s\hgam_i + \cfrac{1}{\gam_i+\blC_{i+1}(s)}}, \quad i=m,m-1,\ldots, 1,\end{eqnarray*}
 We consider three truncating values
 \be
 \blC^{\scalebox{0.5}[0.4]{\rm Gau{\ss} }}_{m+1}(s) = 0 < \blC^{\scalebox{0.5}[0.4]{\rm Nudelman}}_{m+1}(s) = ({\bphi}\sqrt{s})^{-1} < \blC^{\scalebox{0.5}[0.4]{\rm Gau{\ss}--Radau}}_{m+1}(s) = \infty . 
 \ee
From 
\be
\cfrac{1}{s\hgam_i + \cfrac{1}{\gam_i+\blC_{i+1}(s)}} < \cfrac{1}{s\hgam_i + \cfrac{1}{\gam_i+\blC_{i+1}(s)+{\boldsymbol\tau}}}, \text{ for ${\boldsymbol\tau} >0$}
\ee
it directly follows that
\begin{eqnarray*}
\blF_m(s,\blC^{\scalebox{0.5}[0.4]{\rm Gau{\ss} }}_{m+1}(s))< \blF_{m+1}(s,\blC^{\scalebox{0.5}[0.4]{\rm Gau{\ss} }}_{m+2}(s))< \blF_m(s,\blC^{\scalebox{0.5}[0.4]{\rm Nudelman}}_{m+1}(s))\\< 
\blF_{m+1}(s,\blC^{\scalebox{0.5}[0.4]{\rm Gau{\ss}--Radau}}_{m+2}(s)) <
\blF_m(s,\blC^{\scalebox{0.5}[0.4]{\rm Gau{\ss}--Radau}}_{m+1}(s)).
\end{eqnarray*}

}
\end{remark}
\section{Extended Kreĭn--Nudelman string}\label{ap:extended}
As noted before, computation of $\Beta_{m+1}$ at the $m$-th step of the block Lanczos algorithm does not require an additional matrix multiplication, allowing us to compute $\hgam_{m+1} =\hKappa_{m+1}^T \hKappa_{m+1}$ according to Algorithm~\ref{alg:GammaExtraction}. For that,  we extend equation~\eqref {eq:linei} by appending an equation for $i=m+1$ as 
 \be\label{eq:linei+1} 
 \frac 1 {\gam_{m}}\left(\blU_{m+1} -\blU_{m}\right) 
 - \frac 1 {\bxi} \left(\blU_{m+2}-\blU_{m+1}\right)+s\hat\gam_{m+1} \blU_{m+1}=0,
 \ee
replacing \eqref{eq:sommerfeld} with
\be\label{eq:sommerfeld+1}
({\sqrt{s}}{\bphi} ) \blU_{m+2}= - \frac 1 {\bxi} (\blU_{m+2}-\blU_{m+1}),
\ee
 where $\bxi\in\RR^{p\times p}$ is a s.p.d. matrix.

We call the system (\ref{eqn:line1},\ref{eq:linei}, \ref{eq:linei+1},\ref{eq:sommerfeld+1}) the ``Extended Kreĭn--Nudelman string''.
We will define the extended Kreĭn--Nudelman quadrature  as
\[\blF_m^{{\bphi},{\bxi}}=\blU_1,\]
with  $\blU$ satisfying (\ref{eqn:line1},\ref{eq:linei}, \ref{eq:linei+1},\ref{eq:sommerfeld+1}).
 Combining and symmetrizing this matrix pencil, we obtain
 \be\label{eq:KreinNudelman+1} 
 \hat\blF_m^{{\bphi},{\bxi}}(s)=E_1^T(\hat T_m^{{\bphi},{\bxi}}(s)+sI)^{-1}E_1,
 \ee
 where $\hat T_m^{{\bphi},{\bxi}}(s)\in\RR^{p(m+1)\times p(m+1)}$ can be represented as
 \be\label{eq:extT}
 \hat T_m^{{\bphi},{\bxi}}=\begin{pmatrix}
     T_m&\Beta_{m+1}^T\\
     \Beta_{m+1}& \hat \Alpha_{m+1}^{{\bphi},{\bxi}}(s)
 \end{pmatrix},
 \ee
 with
\be\label{eq:hatal+}\hat \Alpha_{m+1}^{{\bphi},{\bxi}}(s)=(\hKappa_{m+1})^{-T}[\gam_m^{-1}+\bxi^{-1}-\bxi^{-1}(\bxi^{-1}+\sqrt{s}\bphi)^{-1}\bxi^{-1}]( \hKappa_{m+1})^{-1}.\ee
Now, it is easy to see that we can obtain Gau{\ss}--Radau approximant $\tilde \blF_{m+1}(s)=\lim_{\bxi\to\infty}\hat\blF_m^{{\bphi},{\bxi}}(s)$, thus 
$\hat \Alpha_{m+1}^{{\bphi},\infty}(s)=(\hKappa_{m+1})^{-T}\gam_m^{-1}(\hKappa_{m+1})^{-1}$ and
\be\label{eq:hatal+e}\delta \Alpha(s)=\hat \Alpha_{m+1}^{{\bphi},\bxi}(s)-\hat \Alpha_{m+1}^{{\bphi},\infty}(s)= (\hKappa_{m+1})^{-T}[\bxi^{-1}-\bxi^{-1}(\bxi^{-1}+\sqrt{s}{\bphi})^{-1}\bxi^{-1}](\hKappa_{m+1})^{-1}\ee
is an s.p.d. matrix if $\bphi,\bxi$ are s.p.d. \, .

\section{Port-Hamiltonian interpretation of the Kre\u{\i}n--Nudelman extension}
\label{sec:port}
{\clr
In this appendix we derive the objective function~\eqref{eq:Optimize} used to
select $\bphi$. We first rewrite the block Stieltjes string~\eqref{eq:LAform}
as a first-order (port-Hamiltonian) system and identify its electric and
magnetic energy. Adding the Kre\u{\i}n--Nudelman termination then produces a
single dissipative port, and the active and reactive power of that port are
read off from $\hat\blF^{\bphi}_m(s)$. We then formulate the objective as one of maximizing active power relative to the stored energy.

Throughout this appendix $s\in\RR_-$, and we fix the branch of the square root
so that \emph{$\sqrt{s}$ is purely imaginary with positive imaginary part}.
Consequently
\be\label{eq:branch}
\overline{\sqrt{s}}=-\sqrt{s},
\qquad
|\sqrt{s}|^{2}=|s|=-s,
\qquad
\frac{1}{\sqrt{s}}=-\frac{\sqrt{s}}{|s|}.
\ee
For a Hermitian positive semi-definite $M\in\CC^{mp\times mp}$ and
$V\in\CC^{mp\times p}$ we write
\be\label{eq:normdef}
\|V\|^{2}_{M}:=\mathrm{tr}\!\left(V^{H}MV\right)=\|M^{1/2}V\|^{2}_{F},
\ee
i.e.\ in the MIMO case the energy of all $p$ states is summed. This is technically only a norm for positive-definite $M$, since there may exist $V$ in the null space of $M$.

\subsection{First-order form}

Let $U=[\blU_1;\dots;\blU_m]$ solve the pencil system~\eqref{eq:LAform},
$(Z_m+s\widehat{\Gamma}_m)U=E_1$ with
$Z_m=J_m\Gamma_m^{-1}J_m^{T}$. Introducing the dual variable
$\hat U=[\hat\blU_1;\dots;\hat\blU_m]$ through
\be\label{eq:dual}
\Gamma_m^{-1}J_m^{T}U=\sqrt{s}\,\hat U,
\quad\text{i.e.}\quad
-\frac{1}{\gam_i}\left(\blU_{i+1}-\blU_i\right)=\sqrt{s}\,\hat\blU_i,
\quad i=1,\dots,m, \text{ with } \blU_{m+1}=\mathbf{0}
\ee
and dividing the first equation by $\sqrt{s}$ gives the first-order form
\be\label{eq:NoLossPH}
\left[
\begin{pmatrix} 0 & J_m\\ -J_m^{T} & 0\end{pmatrix}
+\sqrt{s}
\begin{pmatrix} \widehat{\Gamma}_m & 0\\ 0 & \Gamma_m
\end{pmatrix}
\right]
\begin{pmatrix} U\\ \hat U\end{pmatrix}
=\begin{pmatrix}\dfrac{E_1}{\sqrt{s}}\\[2pt] 0\end{pmatrix},
\ee
where we used $\hKappa_1=I_p$, hence $\widehat{K}_m^{T}E_1=E_1$.
Here $\Gamma_m^{-1}J_m^{T}$ and $-J_m\,$ act as finite-difference
approximations of the derivative, and $\sqrt{s}$ is the Laplace variable of
Remark~\ref{rem:oneLap}. Equation~\eqref{eq:NoLossPH} is a port-Hamiltonian
system: a real skew-symmetric structure matrix plus $\sqrt{s}$ times a real
symmetric positive definite energy matrix, driven through a single port $E_1$.

\noindent It is convenient to keep both the pencil variable $U$ and the variable
\be\label{eq:Xdef}
X=\widehat{K}_mU,
\qquad
(T_m+sI)X=E_1,
\qquad
\blF_m(s)=E_1^{T}X=E_1^{T}U=\blU_1 .
\ee
The system~\eqref{eq:NoLossPH} may be read either as the mass--spring--damper
string of Kre\u{\i}n--Nudelman theory or as an LC transmission line, with
$\blU_i$ the voltage across a capacitor $C_i=\hgam_i$ and $\hat\blU_i$ the
current through an inductor $L_i=\gam_i$. We use the latter for illustration,
see Figure~\ref{fig:tikzexample}.

\begin{figure}[ht]
  \centering
  \includegraphics[width=\textwidth]{\pathfig/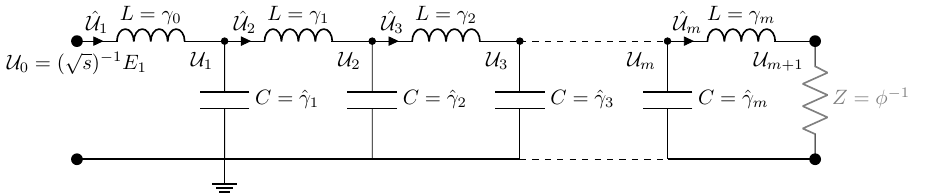}
  \caption{Embedding of the block tridiagonal system~\eqref{eq:Ph_phi} as an LC
  transmission line truncated by a single resistor. We give the values of the
  electrical impedance; equivalently, the line can be built using inductors
  $L_i=\gam_i,\ i=0,\dots,m$, capacitors $C_i=\hgam_i,\ i=1,\dots,m$, and a
  single terminating resistor.}
  \label{fig:tikzexample}
\end{figure}

\subsection{Energy balance of the lossless network (lossless string)}\label{sec:portDEnergy}

Multiplying~\eqref{eq:NoLossPH} from the left by $[U^{H}\ \hat U^{H}]$ and
taking traces gives, with $U^{H}E_1=(E_1^{T}U)^{H}=\blF_m(s)^{H}$,
\be\label{eq:balance0}
\frac{1}{\sqrt{s}}\,\mathrm{tr}\!\left(\blF_m(s)^{H}\right)
=\Big[\mathrm{tr}\big(U^{H}J_m\hat U\big)-\mathrm{tr}\big(\hat U^{H}J_m^{T}U\big)\Big]
+\sqrt{s}\left(\|U\|^{2}_{\widehat{\Gamma}_m}
+\|\hat U\|^{2}_{\Gamma_m}\right).
\ee
The square bracket is the energy exchanged between capacitors and inductors, since
$J_m$ is real it equals $c-\bar c$ with $c=\mathrm{tr}(U^{H}J_m\hat U)$ and is
therefore purely imaginary. The two remaining terms are the stored electric
and magnetic energy,
\be\label{eq:energies0}
2W_e:=\|U\|^{2}_{\widehat{\Gamma}_m}=\|X\|^{2}_{F},
\qquad
2W_m:=\|\hat U\|^{2}_{\Gamma_m}
=\frac{1}{|s|}\,\|J_m^{T}U\|^{2}_{\Gamma_m^{-1}}
=\frac{1}{|s|}\,\|X\|^{2}_{T_m},
\ee
where the first identity uses $\hgam_j=\hKappa_j^{T}\hKappa_j$ (i.e. 
$\widehat{\Gamma}_m=\widehat{K}_m^{T}
\widehat{K}_m$), and the second uses~\eqref{eq:dual},
$|\sqrt{s}|^{2}=|s|$ and the decomposition~\eqref{eq:defT}. The Hamiltonian is
$\mathcal{H}=W_e+W_m$ and the total stored energy is
\be\label{eq:Etot0}
E_{\rm tot}(s)=2\left(W_e+W_m\right)
=\|X\|^{2}_{F}+\frac{1}{|s|}\,\|X\|^{2}_{T_m}.
\ee
Finally, by~\eqref{eq:dual} and~\eqref{eq:branch},
$\mathrm{tr}(U^{H}J_m\hat U)=\mathrm{tr}\big((J_m^{T}U)^{H}\hat U\big)
=\overline{\sqrt{s}}\,\|\hat U\|^{2}_{\Gamma_m}
=-\sqrt{s}\,(2W_m)$, so the bracket in~\eqref{eq:balance0} equals
$-2\sqrt{s}\,(2W_m)$ and
\be\label{eq:balance0final}
\frac{1}{\sqrt{s}}\,\mathrm{tr}\!\left(\blF_m(s)^{H}\right)
=\sqrt{s}\,\big(2W_e-2W_m\big).
\ee
The right-hand side is purely imaginary, i.e. the Dirichlet-terminated string ($\blU_{m+1}=\mathbf{0}$) stores
energy but does not dissipate any. Note that~\eqref{eq:balance0final} contains
only the \emph{difference} $W_e-W_m$, i.e. the reactive power. The total stored
energy~\eqref{eq:Etot0} is not visible in the value of $\blF_m$ at a single $s$,
but it is available from its derivative. For $s\in\RR_-$ away from the poles of
$\blF_m$ the matrix $T_m+sI$ is real and invertible, hence $X$ is real, and
since $T_m$ does not depend on $s$,
\be\label{eq:dF}
\blF_m'(s)=-E_1^{T}(T_m+sI)^{-2}E_1=-X^{T}X,
\qquad\text{so that}\qquad
2W_e=-\mathrm{tr}\,\blF_m'(s).
\ee
Substituting $T_mX=E_1-sX$ into~\eqref{eq:energies0} gives
$2W_m=2W_e+|s|^{-1}\,\mathrm{tr}\,\blF_m(s)$, which is~\eqref{eq:balance0final}
once more, and adding the two contributions yields
\be\label{eq:foster}
E_{\rm tot}(s)=-2\,\mathrm{tr}\,\blF_m'(s)
+\frac{1}{|s|}\,\mathrm{tr}\,\blF_m(s).
\ee

\subsection{The Kre\u{\i}n--Nudelman termination as a dissipative port}

We now replace the Dirichlet condition $\blU_{m+1}=\mathbf{0}$ (see~\eqref{eqn:linem}) by the generalized outflow condition
\be\label{eq:outflow}
f(s)\,\blU^{\bphi}_{m+1}
=-\frac{1}{\gam_m}\left(\blU^{\bphi}_{m+1}-\blU^{\bphi}_{m}\right)
=\sqrt{s}\,\hat\blU^{\bphi}_{m},
\ee
and take $f(s)=\bphi\sqrt{s}$ as in the main text. Eliminating
$\blU^{\bphi}_{m+1}$ from~\eqref{eq:outflow} yields
$\blU^{\bphi}_m=\sqrt{s}\big(\gam_m+(\sqrt{s}\bphi)^{-1}\big)\hat\blU^{\bphi}_m$,
that is, the substitution $\gam_m\mapsto\gam_m+(\sqrt{s}\bphi)^{-1}$. We
therefore introduce the block diagonal matrix
\be\label{eq:GamPhi}
\Gamma^{\bphi}_m(s)
:=\Gamma_m+\frac{1}{\sqrt{s}}E_m\bphi^{-1}E_m^{T},
\ee
which coincides with $\Gamma_m$ except for its last block
$\gam^{\bphi}_m(s)=\gam_m+(\sqrt{s}\bphi)^{-1}$. Since
$\sqrt{s}\,\Gamma^{\bphi}_m
=\sqrt{s}\,\Gamma_m+E_m\bphi^{-1}E_m^{T}$, the
system~\eqref{eq:NoLossPH} becomes
\be\label{eq:Ph_phi}
\left[
\begin{pmatrix} 0 & J_m\\ -J_m^{T} & 0\end{pmatrix}
+\begin{pmatrix} 0 & 0\\ 0 & E_m\bphi^{-1}E_m^{T}\end{pmatrix}
+\sqrt{s}
\begin{pmatrix} \widehat{\Gamma}_m & 0\\ 0 & \Gamma_m
\end{pmatrix}
\right]
\begin{pmatrix} U^{\bphi}\\ \hat U^{\bphi}\end{pmatrix}
=\begin{pmatrix}\dfrac{E_1}{\sqrt{s}}\\[2pt] 0\end{pmatrix}.
\ee
This is the standard dissipative port-Hamiltonian form \cite{PortHammelVanDer,PortHammilBettie}, the real skew-symmetric
structure matrix is unchanged and a real symmetric positive semi-definite
resistive term $E_m\bphi^{-1}E_m^{T}$ of rank at most $p$ is added at the
truncation.

Eliminating $\hat U^{\bphi}=\frac{1}{\sqrt{s}}
(\Gamma^{\bphi}_m)^{-1}J_m^{T}U^{\bphi}$ recovers the
Kre\u{\i}n--Nudelman operator of the main text,
\be\label{eq:Tphi}
\hat T^{\bphi}_m(s)
=\widehat{K}_m^{-T}J_m
\big(\Gamma^{\bphi}_m(s)\big)^{-1}
J_m^{T}\widehat{K}_m^{-1},
\qquad
\big(\hat T^{\bphi}_m(s)+sI\big)X^{\bphi}=E_1,
\ee
with $\hat\blF^{\bphi}_m(s)=E_1^{T}X^{\bphi}=E_1^{T}U^{\bphi}
=\blU^{\bphi}_1$. Because $\Gamma^{\bphi}_m$ is block diagonal and
differs from $\Gamma_m$ only in its last block, $\hat T^{\bphi}_m$
differs from $T_m$ only in its last diagonal block and the Sherman–Morrison–Woodbury identity
\be\label{eq:woodbury}
\left(\gam_m+(\sqrt{s}\bphi)^{-1}\right)^{-1}
=\gam_m^{-1}-\gam_m^{-1}\left(\gam_m^{-1}+\sqrt{s}\bphi\right)^{-1}\gam_m^{-1}
\ee
shows that this block is exactly $\hat\Alpha^{\bphi}_m=\Alpha_m-\delta\Alpha_m$
of~\eqref{eq:hatal}. Thus~\eqref{eq:Ph_phi} is a first-order realization
of~\eqref{eq:KreinNudelman}.

Since $\Gamma^{\bphi}_m$ is complex symmetric we split it into
Hermitian parts. By the branch definition of $\sqrt(s)$~\eqref{eq:branch},
\[
\tfrac{1}{2}\big(\Gamma^{\bphi}_m
+(\Gamma^{\bphi}_m)^{H}\big)=\Gamma_m \text{ and }
\tfrac{1}{2}\big((\Gamma^{\bphi}_m)^{H}
-\Gamma^{\bphi}_m\big)
=\frac{\sqrt{s}}{|s|}E_m\bphi^{-1}E_m^{T},\] 
so that with
\[
M^{-1}\pm M^{-H}=M^{-H}(M^{H}\pm M)M^{-1}
\]
we obtain
\be\label{eq:GamSplit}
\Re\big[(\Gamma^{\bphi}_m)^{-1}\big]
=(\Gamma^{\bphi}_m)^{-H}\Gamma_m
(\Gamma^{\bphi}_m)^{-1}>0,
\quad
\Im\big[(\Gamma^{\bphi}_m)^{-1}\big]
=\frac{(\Gamma^{\bphi}_m)^{-H}E_m\bphi^{-1}E_m^{T}
(\Gamma^{\bphi}_m)^{-1}}{\sqrt{|s|}}\geq0 .
\ee
As $\widehat{K}_m$ and $J_m$ are real, the same splitting carries
over to~\eqref{eq:Tphi},
\be\label{eq:TSplit}
\Re\,\hat T^{\bphi}_m
=\widehat{K}_m^{-T}J_m
\Re\big[(\Gamma^{\bphi}_m)^{-1}\big]J_m^{T}
\widehat{K}_m^{-1}>0,
\qquad
\Im\,\hat T^{\bphi}_m
=\widehat{K}_m^{-T}J_m
\Im\big[(\Gamma^{\bphi}_m)^{-1}\big]J_m^{T}
\widehat{K}_m^{-1}\geq0,
\ee
where $\Re\,\hat T^{\bphi}_m\to T_m$ and $\Im\,\hat T^{\bphi}_m\to0$ as
$\bphi^{-1}\to0$. The dissipative part $\Im\,\hat T^{\bphi}_m$ has rank at most
$p$ and is supported at the truncation.

\subsection{Active and reactive power}

Repeating the computation for the earlier section~\ref{sec:portDEnergy} for~\eqref{eq:Ph_phi}, the
only change is in the exchange term: now
$J_m^{T}U^{\bphi}=\sqrt{s}\,\Gamma_m\hat U^{\bphi}
+E_m\bphi^{-1}E_m^{T}\hat U^{\bphi}$, so
\be\label{eq:cphi}
\mathrm{tr}\big((U^{\bphi})^{H}J_m\hat U^{\bphi}\big)
=-\sqrt{s}\,\|\hat U^{\bphi}\|^{2}_{\Gamma_m}
+\|\hat\blU^{\bphi}_m\|^{2}_{\bphi^{-1}} .
\ee
Writing $2W^{\bphi}_e=\|U^{\bphi}\|^{2}_{\widehat{\Gamma}_m}$,
$2W^{\bphi}_m=\|\hat U^{\bphi}\|^{2}_{\Gamma_m}$ and
$P^{\bphi}=\|\hat\blU^{\bphi}_m\|^{2}_{\bphi^{-1}}$, all real, and using
$\overline{\sqrt{s}}=-\sqrt{s}$, the balance~\eqref{eq:balance0} can be summarized as
\be\label{eq:master}
\frac{1}{\sqrt{s}}\,\mathrm{tr}\!\left(\big(\hat\blF^{\bphi}_m(s)\big)^{H}\right)
=P^{\bphi}(s)+\sqrt{s}\,\Big(2W^{\bphi}_e(s)-2W^{\bphi}_m(s)\Big).
\ee
Because $\sqrt{s}$ is purely imaginary, \eqref{eq:master} separates the
\emph{active} (dissipated) power from the \emph{reactive} power,
\be\label{eq:active}
\Re\,\mathrm{tr}\!\left(
\frac{\big(\hat\blF^{\bphi}_m(s)\big)^{H}}{\sqrt{s}}\right)
=P^{\bphi}(s)
=\big\|\hat\blU^{\bphi}_m\big\|^{2}_{\bphi^{-1}}\ \ge 0,
\ee
\be\label{eq:reactive}
\Im\,\mathrm{tr}\!\left(
\frac{\big(\hat\blF^{\bphi}_m(s)\big)^{H}}{\sqrt{s}}\right)
=\sqrt{|s|}\,\Big(2W^{\bphi}_e(s)-2W^{\bphi}_m(s)\Big),
\ee
and~\eqref{eq:balance0final} is recovered for $\bphi^{-1}=0$. All three
quantities are computable from $X^{\bphi}$ substituting
$\sqrt{s}\,\hat U^{\bphi}=(\Gamma^{\bphi}_m)^{-1}J_m^{T}U^{\bphi}$
into the definitions and using~\eqref{eq:GamSplit}--\eqref{eq:TSplit},
\be\label{eq:normforms}
2W^{\bphi}_e=\|X^{\bphi}\|^{2}_{F},
\qquad
2W^{\bphi}_m=\frac{1}{|s|}\,\|X^{\bphi}\|^{2}_{\Re\,\hat T^{\bphi}_m(s)},
\qquad
P^{\bphi}=\frac{1}{\sqrt{|s|}}\,\|X^{\bphi}\|^{2}_{\Im\,\hat T^{\bphi}_m(s)} .
\ee
In particular the electric energy form is unaffected by the termination, the
magnetic energy sees only $\Re\,\hat T^{\bphi}_m$, and the dissipation is
carried entirely by $\Im\,\hat T^{\bphi}_m$. The total stored energy is
\be\label{eq:EtotPhi}
E^{\bphi}_{\rm tot}(s)=2\left(W^{\bphi}_e+W^{\bphi}_m\right)
=\|X^{\bphi}\|^{2}_{F}
+\frac{1}{|s|}\,\|X^{\bphi}\|^{2}_{\Re\,\hat T^{\bphi}_m(s)} .
\ee
Moreover, no analogue of~\eqref{eq:foster}, is available once $\bphi>0$. Since
$X^{\bphi}$ is a rational function of $\sqrt{s}$ with real coefficients, the
branch~\eqref{eq:branch} gives
$\overline{X^{\bphi}(\sqrt{s})}=X^{\bphi}(-\sqrt{s})$, so the electric energy
\be\label{eq:twopoint}
\|X^{\bphi}\|^{2}_{F}
=\big(X^{\bphi}(-\sqrt{s})\big)^{T}X^{\bphi}(\sqrt{s})
\ee
has the solution at $+\sqrt{s}$ and its conjucate, i.e. the one at $-\sqrt{s}$, whereas
differentiation of $\hat\blF^{\bphi}_m$ produces only the unconjugated form
$(X^{\bphi})^{T}(\sqrt{s})X^{\bphi}(\sqrt{s})$. The two coincide in the lossless case, where $X$ is
real, but not otherwise. The magnetic energy, by contrast, does follow from the
transfer function, substituting $\hat T^{\bphi}_mX^{\bphi}=E_1-sX^{\bphi}$ into
$\|X^{\bphi}\|^{2}_{\Re\hat T^{\bphi}_m}
=\Re\,\mathrm{tr}\big((X^{\bphi})^{H}\hat T^{\bphi}_mX^{\bphi}\big)$ gives
\be\label{eq:WmTF}
2W^{\bphi}_m(s)=\|X^{\bphi}\|^{2}_{F}
+\frac{\Re\,\mathrm{tr}\,\hat\blF^{\bphi}_m(s)}{|s|},
\ee
so that $\|X^{\bphi}\|^{2}_{F}$ is the only quantity in~\eqref{eq:normforms}
that must be computed from the state.

\subsection{The objective}

Maximizing the dissipation of the truncated string relative to the energy it
stores gives the objective used in section~\ref{sec:HermitePade},
\be\label{eq:OptimizeA}
\bphi^{\star}
=\arg\max_{\bphi\geq0}\ \int_{\cal S}
\frac{\Re\,\mathrm{tr}\!\left(\dfrac{\big(\hat\blF^{\bphi}_m(s)\big)^{H}}
{\sqrt{s}}\right)}
{\|X^{\bphi}\|^{2}_{F}
+\dfrac{1}{|s|}\,\|X^{\bphi}\|^{2}_{\Re\,\hat T^{\bphi}_m(s)}}\ ds ,
\ee
i.e.\ the active power absorbed by the Kre\u{\i}n--Nudelman port, normalized by
the total stored energy. Both the numerator and the denominator are real and
non-negative by~\eqref{eq:active} and~\eqref{eq:EtotPhi}.

\begin{remark}\label{rem:limits}
The objective vanishes at both quadrature limits. For $\bphi\to\infty$ we have the block Gau{\ss} rule and for $\bphi\to0$ the block Gau{\ss}--Radau rule, both without absorbed power. The maximizer
is therefore attained in the interior.
\end{remark}
}

\end{appendices}

\bibliography{sn-bibliography}

@book{PortHammelVanDer,
  title     = {Port-{H}amiltonian Systems Theory: An Introductory Overview},
  author    = {{van der Schaft}, Arjan and Jeltsema, Dimitri},
  year      = {2014},
  isbn      = {978-1-60198-786-0},
  series    = {Foundations and Trends in Systems and Control},
publisher = {Now Publishers Inc.},
address = {Hanover, MA, USA},
volume = {1},
  number    = {2-3}
}

@article{PortHammilBettie,
title = {Linear port-{H}amiltonian descriptor systems},
author = {Christopher Beattie and Volker Mehrmann and Hongguo Xu and Hans Zwart},
year = {2018},
month = {dec},
day = {1},
language = {English},
volume = {30},
journal = {Mathematics of Control, Signals, and Systems},
issn = {0932-4194},
publisher = {Springer},
number = {4},
}

@article{Krein1947,
  author = {Krein, M. G.},
  title = {The theory of extensions of semi-bounded {H}ermitian operators and its applications},
  journal = {Matematicheskii Sbornik},
  volume = {20},
  pages = {431--495},
  year = {1947},
  note = {Part I; Part II in vol. 21 (1947), pp. 365–404}
}

@article{Krein1952,
  author = {Krein, M. G.},
  title = {On a generalization of {S}tieltjes' investigations},
  journal = {Doklady Akademii Nauk SSSR},
  volume = {86},
  number = {6},
  pages = {881--884},
  year = {1952}
}

@article{Krein1967,
  author = {Krein, M. G.},
  title = {The description of solutions of the truncated moment problem},
  journal = {Matematicheskie Issledovaniya},
  volume = {2},
  number = {2},
  pages = {114--132},
  year = {1967}
}

@article{LagariasOpt,
author = {Lagarias, Jeffrey C. and Reeds, James A. and Wright, Margaret H. and Wright, Paul E.},
title = {Convergence Properties of the {N}elder--{M}ead Simplex Method in Low Dimensions},
journal = {SIAM Journal on Optimization},
volume = {9},
number = {1},
pages = {112-147},
year = {1998},
doi = {10.1137/S1052623496303470}}

@book{KreinNudelman1973,
  author = {Krein, M. G. and Nudelman, A. A.},
  title = {The {M}arkov Moment Problem and Extremal Problems},
  publisher = {Nauka},
  address = {Moscow},
  year = {1973},
  language = {Russian},
  note = {English translation: Translation of Mathematical Monographs, AMS, vol. 50, 1977}
}

@article{KreinNudelman1989,
  author = {Krein, M. G. and Nudelman, A. A.},
  title = {Some spectral properties of a nonhomogeneous string with a dissipative boundary condition},
  journal = {Journal of Operator Theory},
  volume = {22},
  number = {2},
  pages = {369--395},
  year = {1989}
}

@article{RRT16,
title = {New block quadrature rules for the approximation of matrix functions},
journal = {Linear Algebra and its Applications},
volume = {502},
pages = {299-326},
year = {2016},
issn = {0024-3795},
author = {Lothar Reichel and Giuseppe Rodriguez and Tunan Tang}
}

@article{Aptekarev,
  title={{H}ermite-{P}ad{\'e} Approximants for a Pair of {C}auchy Transforms with Overlapping Symmetric Supports},
  author={Aptekarev, Alexander I and Van Assche, Walter and Yattselev, Maxim L},
  journal={Communications on Pure and Applied Mathematics},
  volume={70},
  number={3},
  pages={444--510},
  year={2017},
  publisher={Wiley Online Library}
}

@misc{Suetin,
      title={Rational {H}ermite-{P}ad\'e Approximants vs {P}ad\'e Approximants. Numerical Results}, 
      author={Egor O. Dobrolyubov and Nikolay R. Ikonomov and Leonid A. Knizhnerman and Sergey P. Suetin},
      year={2023},
      eprint={2306.07063},
      archivePrefix={arXiv},
      primaryClass={math.CV},
      url={https://arxiv.org/abs/2306.07063}, 
}

@book{Zworski,
  title={Mathematical Theory of Scattering Resonances},
  author={Dyatlov, Semyon and Zworski, Maciej},
  year={2019},
  publisher={American Mathematical Society},
  address={Providence, RI},
  isbn={978-1-4704-4366-5},
  series={Graduate Studies in Mathematics}
}

@phdthesis{Lun18,
  title        = {A New Block {Krylov} Subspace Framework with
Applications to Functions of Matrices Acting on Multiple Vectors},
  author       = {Kathryn Lund},
  year         = {2018},
  month        = {May},
  address      = {Philadelphia, Pennsylvania, USA},
  school       = {Department of Mathematics, Temple University
and Fakult\"at Mathematik und Naturwissenschaften der Bergischen
Universit\"at Wuppertal},
  type         = {PhD thesis}
}

@book{GM10,
  author    = {Golub, Gene H. and Meurant, G{\'e}rard},
  title     = {Matrices, Moments and Quadrature with Applications},
  publisher = {Princeton University Press},
  address   = {Princeton, NJ},
  year      = {2010},
  isbn      = {978-0-691-14341-5}
}

@misc{Crameri2023ScientificColormapsPackage,
  author       = {Crameri, Fabio},
  title        = {Scientific colour maps},
  month        = {oct},
  year         = {2023},
  publisher    = {Zenodo},
  version      = {v8.0.1},
  doi          = {10.5281/zenodo.8409685},
  url          = {https://doi.org/10.5281/zenodo.8409685}
}

@incollection{Golub1977,
  author    = {Golub, Gene H. and Underwood, Richard},
  title     = {The Block {L}anczos Method for Computing Eigenvalues},
  booktitle = {Mathematical Software {III}},
  editor    = {Rice, John R.},
  year      = {1977},
  pages     = {361--377},
  publisher = {Academic Press},
  address   = {New York}
}

@article{idkGrids,
author = {Ingerman, David and Druskin, Vladimir and Knizhnerman, Leonid},
title = {Optimal finite difference grids and rational approximations of the square root I. {Elliptic} problems},
journal = {Communications on Pure and Applied Mathematics},
volume = {53},
number = {8},
pages = {1039-1066},
year = {2000}
}

@Article{Meurant2023,
author={Meurant, G{\'e}rard
and Tich{\'y}, Petr},
title={The behavior of the {Gauss-Radau} upper bound of the error norm in {CG}},
journal={Numerical Algorithms},
year={2023},
month={Oct},
day={01},
volume={94},
number={2},
pages={847-876},
issn={1572-9265}
}

@article{Greenbaum1989,
title = {Behavior of slightly perturbed {Lanczos} and conjugate-gradient recurrences},
journal = {Linear Algebra and its Applications},
volume = {113},
pages = {7-63},
year = {1989},
issn = {0024-3795},
author = {Anne Greenbaum}
}

@article{lot2013,
author = {Caterina Fenu and David Martin and Reichel, Lothar and Rodriguez, Giuseppe},
title = {Block {Gauss} and Anti-{Gauss} Quadrature with Application to Networks},
journal = {SIAM Journal on Matrix Analysis and Applications},
volume = {34},
number = {4},
pages = {1655-1684},
year = {2013}
}

@article{Zimmerling_KreinNudel,
year = {2023},
month = {dec},
publisher = {IOP Publishing},
volume = {40},
number = {2},
pages = {025002},
author = {Zimmerling, Jörn and Druskin, Vladimir and Guddati, Murthy and Cherkaev, Elena and Remis, Rob},
title = {Solving inverse scattering problems via reduced-order model embedding procedures},
journal = {Inverse Problems},
}

@article{lot2008,
title = {Matrices, moments, and rational quadrature},
journal = {Linear Algebra and its Applications},
volume = {429},
number = {10},
pages = {2540-2554},
year = {2008},
note = {Special Issue in honor of Richard S. Varga},
issn = {0024-3795},
author = {Guillermo {López Lagomasino} and Lothar Reichel and Lena Wunderlich}
}

@inproceedings{amsel2023nearoptimal,
      author={Amsel, Noah and Chen, Tyler and Greenbaum, Anne and Musco, Cameron and Musco, Chris},
    title={Near-Optimal Approximation of Matrix Functions by the {Lanczos} Method},
    booktitle = {Proceedings of the Thirty-Eighth Annual Conference on Neural Information Processing},
    year = {2024}
}

@book{MeurantBook,
author = {Meurant, Gérard and Tichý, Petr},
title = {Error Norm Estimation in the Conjugate Gradient Algorithm},
publisher = {Society for Industrial and Applied Mathematics},
year = {2024},
address = {Philadelphia, PA},
edition   = {}
}

@article{OLeary1980,
title = {The block conjugate gradient algorithm and related methods},
journal = {Linear Algebra and its Applications},
volume = {29},
pages = {293-322},
year = {1980},
note = {Special Volume Dedicated to Alson S. Householder},
issn = {0024-3795},
author = {Dianne P. O'Leary}
}

@article{Druskin2016,
author = {Druskin, Vladimir and Mamonov, Alexander V. and Thaler, Andrew E. and Zaslavsky, Mikhail},
title = {Direct, Nonlinear Inversion Algorithm for Hyperbolic Problems via Projection-Based Model Reduction},
journal = {SIAM Journal on Imaging Sciences},
volume = {9},
number = {2},
pages = {684-747},
year = {2016},
doi = {10.1137/15M1039432}
}

@article{Knizhnerman1996TheSL,
  title={The simple {Lanczos} procedure: estimates of the error of the {G}auss quadrature formula and their applications},
  author={Leonid A. Knizhnerman},
  journal={Computational Mathematics and Mathematical Physics},
  year={1996},
  volume={36},
  pages={1481-1492}
}

@article{druskin1988spectral,
  title={A spectral semi-discrete method for numerical solution of {3D} non-stationary problems in electrical prospecting},
  author={Vladimir Druskin and Leonid Knizhnerman},
  journal={Phys. Sol. Earth},
  volume={24},
  pages={641--648},
  year={1988}
}

@article{druskin2016near,
  title={Near-optimal perfectly matched layers for indefinite {Helmholtz} problems},
  author={Druskin, Vladimir and G\"uttel, Stefan and Knizhnerman, Leonid},
  journal={SIAM Review},
  volume={58},
  number={1},
  pages={90--116},
  year={2016},
  publisher={Society for Industrial and Applied Mathematics}
}

@inproceedings{saputra2024adaptive,
  title={Adaptive Multidimensional Inversion for Borehole Ultra-Deep Azimuthal Resistivity},
  author={Saputra, Wardana and Ambia, Joaquin and Torres-Verd{\'\i}n, Carlos and Davydycheva, Sofia and Druskin, Vladimir and Zimmerling, J{\"o}rn},
  booktitle={SPWLA Annual Logging Symposium},
  pages={D041S013R005},
  year={2024},
  organization={SPWLA}
}

@article{ZaslavskySfraction,
author = {Druskin, Vladimir and Mamonov, Alexander V. and Zaslavsky, Mikhail},
title = {Multiscale {S}-Fraction Reduced-Order Models for Massive Wavefield Simulations},
journal = {Multiscale Modeling \& Simulation},
volume = {15},
number = {1},
pages = {445-475},
year = {2017}
}

@article{druskin1989two,
  title={Two polynomial methods of calculating functions of symmetric matrices},
  author={Druskin, Vladimir L and Knizhnerman, Leonid A},
  journal={USSR Computational Mathematics and Mathematical Physics},
  volume={29},
  number={6},
  pages={112--121},
  year={1989},
  publisher={No longer published by Elsevier}
}

@article{zimmerling2025monotonicity,
  title={Monotonicity, Bounds and Acceleration of Block {G}auss and {G}auss--{R}adau Quadrature for Computing {$B^T \phi(A) B$}},
  author={Zimmerling, J{\"o}rn and Druskin, Vladimir and Simoncini, Valeria},
  journal={Journal of Scientific Computing},
  volume={103},
  number={1},
  pages={5},
  year={2025},
  publisher={Springer US New York}
}

@article{zimmerling2025targeted,
  title={A Targeted Quadrature Framework for Simulating Large-Scale {3D} Anisotropic Electromagnetic Measurements},
  author={Zimmerling, J{\"o}rn and Druskin, Vladimir and Davydycheva, Sofia and Saputra, Wardana and Torres-Verd{\'\i}n, Carlos and Antonsen, Frank and Lotsberg, Jon K{\aa}re and Rabinovich, Michael},
  journal={arXiv preprint arXiv:2511.17999},
  year={2025}
}

@article{UbaruChenSaad2017SLQ,
  author  = {Ubaru, Shashanka and Chen, Jie and Saad, Yousef},
  title   = {Fast Estimation of {$\mathrm{tr}(f(A))$} via Stochastic {Lanczos} Quadrature},
  journal = {SIAM Journal on Matrix Analysis and Applications},
  year    = {2017},
  doi     = {10.1137/16M1104974},
}

@article{AvronToledo2011RandomizedTrace,
  author  = {Avron, Haim and Toledo, Sivan},
  title   = {Randomized Algorithms for Estimating the Trace of an Implicit Symmetric Positive Semi-Definite Matrix},
  journal = {Journal of the ACM},
  volume  = {58},
  number  = {2},
  year    = {2011},
  doi     = {10.1145/1944345.1944349}
}

@article{ColbrookHorningTownsend2021SIAMRevSpectralMeasures,
  author  = {Colbrook, Matthew J. and Horning, Andrew and Townsend, Alex},
  title   = {Computing Spectral Measures of Self-Adjoint Operators},
  journal = {SIAM Review},
  volume  = {63},
  number  = {3},
  year    = {2021},
  pages   = {489--524},
  doi     = {10.1137/20M1333985}
}

@article{YiMassattHorningLuskinPixleyKaye2025DeltaChebyshev,
  author  = {Yi, Jinjing and Massatt, Daniel and Horning, Andrew and Luskin, Mitchell and Pixley, J. H. and Kaye, Jason},
  title   = {A High-Order Regularized Delta-{Chebyshev} Method for Computing Spectral Densities},
  journal = {arXiv},
  year    = {2025},
  eprint  = {2512.03149},
  archivePrefix = {arXiv},
  primaryClass  = {math.NA},
  doi     = {10.48550/arXiv.2512.03149}
}

@incollection{BeerPettifor1984NATOASI,
  author    = {Beer, N. and Pettifor, D. G.},
  title     = {The Recursion Method and the Estimation of Local Densities of States},
  booktitle = {The Electronic Structure of Complex Systems},
  editor    = {Phariseau, P. and Temmerman, W. M.},
  series    = {NATO ASI Series},
  volume    = {113},
  pages     = {769--777},
  publisher = {Springer},
  address   = {Boston, MA},
  year      = {1984},
  doi       = {10.1007/978-1-4613-2405-8_14}
}

@article{HaydockHeineKelly1972,
  author  = {Haydock, Roger and Heine, Volker and Kelly, M. J.},
  title   = {Electronic structure based on the local atomic environment},
  journal = {Journal of Physics C: Solid State Physics},
  volume  = {5},
  year    = {1972},
  pages   = {2845--2858}
}

@article{Haydock1980Recursion,
  author    = {Haydock, Roger},
  title     = {The recursion method},
  journal   = {Solid State Physics},
  volume    = {35},
  year      = {1980},
  pages     = {215--294},
  publisher = {Academic Press}
}

@incollection{BeerPettifor1984Terminator,
  author    = {Beer, N. and Pettifor, D. G.},
  title     = {Terminators for continued-fraction {Green's} functions in the recursion method},
  booktitle = {Electronic Structure of Complex Systems},
  editor    = {Phariseau, P. and Temmerman, W. M.},
  publisher = {Plenum Press},
  address   = {New York},
  year      = {1984},
  pages     = {769--777}
}

@article{LuchiniNex1987Stitching,
  author  = {Luchini, M. U. and Nex, C. M. M.},
  title   = {A new procedure for appending terminators in the recursion method},
  journal = {Journal of Physics C: Solid State Physics},
  volume  = {20},
  year    = {1987},
  pages   = {3125--3130}
}

@article{PinnaLuntvonKeyserlingk2025Stitching,
  title = {Approximation theory for {G}reen's functions via the {L}anczos algorithm},
  author = {Pinna, Gabriele and Lunt, Oliver and von Keyserlingk, Curt},
  journal = {Phys. Rev. B},
  volume = {112},
  issue = {5},
  pages = {054435},
  numpages = {25},
  year = {2025},
  month = {Aug},
  publisher = {American Physical Society}
}

@book{PettiforWeaire1985,
  editor    = {Pettifor, D. G. and Weaire, D. L.},
  title     = {The Recursion Method and Its Applications},
  series    = {Springer Series in Solid-State Sciences},
  volume    = {58},
  publisher = {Springer-Verlag},
  address   = {Berlin, Heidelberg, New York, Tokyo},
  year      = {1985},
  isbn      = {3-540-15173-7, 978-3-540-15173-9}
}

@article{PerssonChenMusco2025,
title = {Randomized block-{Krylov} subspace methods for low-rank approximation of matrix functions},
journal = {Linear Algebra and its Applications},
volume = {741},
pages = {32-65},
year = {2026},
issn = {0024-3795},
author = {David Persson and Tyler Chen and Christopher Musco}
}

@article{Reichel2020SimpleSC,
  title={Simple stopping criteria for the {LSQR} method applied to discrete ill-posed problems},
  author={Lothar Reichel and Hassane Sadok and Wei-Hong Zhang},
  journal={Numerical Algorithms},
  year={2020},
  volume={84},
  pages={1381 - 1395},
  url={https://api.semanticscholar.org/CorpusID:209429761}
}

@article{Fasondini2019QuadraticPade,
  author  = {Fasondini, M. and Hale, N. and Spoerer, R. and Weideman, J. A. C.},
  title   = {Quadratic {Pad{\'e}} Approximation: Numerical Aspects and Applications},
  journal = {Computer Research and Modeling},
  volume  = {11},
  number  = {6},
  pages   = {1017--1031},
  year    = {2019}
}

@article{MayerNuttallTong1984QuadraticPade,
  author  = {Mayer, I. L. and Nuttall, J. and Tong, B. Y.},
  title   = {Quadratic {Pad{\'e}} Approximant Method for Calculating Densities of States},
  journal = {Physical Review B},
  volume  = {29},
  number  = {12},
  pages   = {7102},
  year    = {1984}
}

@article{MayerTong1985QuadraticPade,
  author  = {Mayer, I. L. and Tong, B. Y.},
  title   = {The Quadratic {Pad{\'e}} Approximant Method and Its Application for Calculating Densities of States},
  journal = {Journal of Physics C: Solid State Physics},
  volume  = {18},
  number  = {17},
  pages   = {3297--3318},
  year    = {1985}
}

@inproceedings{ChenEpperlyMeyerMuscoRao2026,
  author    = {Chen, Tyler and Epperly, Ethan N. and Meyer, Raphael A.
               and Musco, Christopher and Rao, Akash},
  title     = {Does Block Size Matter in Randomized Block {Krylov}
               Low-Rank Approximation?},
  booktitle = {Proceedings of the 2026 Annual ACM-SIAM Symposium
               on Discrete Algorithms (SODA)},
  year      = {2026},
  note      = {arXiv:2508.06486}
}

@article{ColbrookHansen2023,
  author  = {Colbrook, Matthew J. and Hansen, Anders C.},
  title   = {The Foundations of Spectral Computations
             via the Solvability Complexity Index Hierarchy},
  journal = {Journal of the European Mathematical Society},
  volume  = {25},
  number  = {12},
  pages   = {4639--4728},
  year    = {2023}
}

@article{BeckermannGuttel2010,
  author  = {Beckermann, Bernhard and G{\"u}ttel, Stefan},
  title   = {On the Convergence of Rational {Ritz} Values},
  journal = {SIAM Journal on Matrix Analysis and Applications},
  volume  = {31},
  number  = {4},
  pages   = {1740--1774},
  year    = {2010},
  doi     = {10.1137/090755412}
}

@article{Kuijlaars2000,
  author  = {Kuijlaars, Arno B. J.},
  title   = {Which Eigenvalues Are Found by the {Lanczos} Method?},
  journal = {SIAM Journal on Matrix Analysis and Applications},
  volume  = {22},
  number  = {1},
  pages   = {306--321},
  year    = {2000}
}

@book{GantmacherKrein2002,
  author    = {Gantmacher, F. R. and Krein, M. G.},
  title     = {Oscillation Matrices and Kernels and Small Vibrations of Mechanical Systems},
  publisher = {American Mathematical Society},
  address   = {Providence, RI},
  year      = {2002},
  series    = {AMS Chelsea Publishing},
  note      = {Revised edition. See Supplement II: ``On a Remarkable Problem for a String with Beads and Continued Fractions of Stieltjes''}
}

@book{LiesenStrakos2012,
  author    = {Liesen, J{\"o}rg and Strako{\v{s}}, Zden{\v{e}}k},
  title     = {{Krylov} Subspace Methods: Principles and Analysis},
  publisher = {Oxford University Press},
  address   = {Oxford},
  year      = {2012}
}

\end{document}